\documentclass[a4paper, 10pt, conference]{article}     

\usepackage{amssymb}
\usepackage{color}
\usepackage{caption}
\usepackage{subcaption}
\usepackage{mathtools}
\mathtoolsset{showonlyrefs=true}

\usepackage{amsfonts}
\usepackage{graphicx}%
\providecommand{\U}[1]{\protect\rule{.1in}{.1in}}
\newtheorem{theorem}{Theorem}
\newtheorem{definition}{Definition}
\newtheorem{remark}{Remark}
\newtheorem{proposition}{Proposition}
\newenvironment{proof}[1][Proof]{\noindent\textbf{#1.} }{\ \rule{0.5em}{0.5em}}

\usepackage{amssymb}
\usepackage{color}
\usepackage{caption}
\usepackage{subcaption}
\usepackage{mathtools}
\usepackage{soul}
\mathtoolsset{showonlyrefs=true}

\usepackage{amsfonts}
\usepackage{wrapfig}
\usepackage{graphicx}%
\providecommand{\U}[1]{\protect\rule{.1in}{.1in}}

\newcommand{\e}{\varepsilon}

\newcommand{\R}{\mathbb{R}}
\newcommand{\B}{\mathbb{B}}
\newcommand{\T}{\mathcal{T}}
\newcommand{\D}{\mathcal{D}}
\newcommand{\A}{\mathcal{A}}

\title{\LARGE \bf
Hamilton-Jacobi-Bellman Equation for Control Systems with Friction 
}
\author{Fabio Tedone$^{*}$ and Michele Palladino$^{*}$
\thanks{$^{*}$GSSI - Gran Sasso Science Institute,
via F. Crispi 7, L'Aquila, 67100, Italy.
{\tt\small fabio.tedone@gssi.it, michele.palladino@gssi.it}}
}
\begin{document}
\maketitle
\thispagestyle{empty}
\pagestyle{empty}
\begin{abstract}
This paper proposes a new framework to model control systems in which a dynamic friction occurs. The model consists of a controlled differential inclusion with a discontinuous right-hand side, which still preserves existence and uniqueness of the solution for each given input function $u(t)$. Under general hypotheses, we are able to derive
the Hamilton-Jacobi-Bellman equation for the related free time optimal control problem and to characterise the value function as the unique, locally Lipschitz continuous viscosity solution.
\end{abstract}

\section{Introduction}
Dynamic friction occurs between two or more solid bodies that are moving one relative to the other and rub together along parts of their surfaces. Modeling dynamic friction is not an easy task since it concerns the study of, possibly discontinuous, dynamic equations which inherit a dissipative structure. A classic approach to modeling the static and the dynamic friction is provided by the Coulomb model \cite{coulomb1821theorie}, which indeed consists in a non-smooth, dissipative dynamics. Since Coulomb's seminal work, several other models for friction have been developed. On this subject, an interesting overview is provided in  \cite{pennestri2016review}, in which several other models for friction are shown. 

Modeling dynamic friction in control systems has so far received much less attention. Although the resulting control system has a discontinuous dynamic equation, its dissipative structure still yields well-posedness of the control system: for a given input $u(t)$ and a given initial condition $x(t_0)=x_0$ at time $t_0$, the system has a unique state $x(t)$ for all $t\geq t_0$. To our knowledge, the only studies on control systems with friction (see, e.g., \cite{colombo2016minimum}, \cite{colombo2020optimization}, \cite{de2019optimal} and references therein)  concern (or are related to) the Moreau's Sweeping Process \cite{moreau1977evolution}, \cite{moreau1999numerical}, which is a notable example of dynamical systems, in which the dynamic friction phenomenon occurs between a rigid body and a moving, perfectly indeformable, active constraint. 

In this paper we will introduce a new framework for modeling the dynamic friction, allowing for a slight penetration of a rigid body into another body (case which happens, for instance in the \emph{rigid-body penetration} field \cite{jones2003one,shi2014model}).
 To motivate the model of study, let us first assume that a solid body $\mathcal{B}$ is partially or totally immersed into another external body. Let $S$ be the region of contact of $\mathcal{B}$ with the external body. 
We aim at deriving the friction produced at a point $x\in S$, when a vector field $g$ is applied to $\mathcal{B}$ at $x$.  Now, suppose that the family of normal vectors to $\mathcal{B}$ is described by the mapping $ \alpha\mapsto \eta(x)\cdot Q(\alpha)$, where $\eta(x)$ is the normal to $\mathcal{B}$ at $x$ and $\alpha \mapsto Q(\alpha)$ is a matrix transporting $\eta(x)$ along $S$. Then, one can approximate the resulting vector field at the point $x\in S$ as the vector field $g$ minus the ``averaged friction" at $x$, namely (see Figure \ref{fig1})
$$g(t,x)-\int_A k(g(t,x),\alpha)\eta(x)\cdot Q(\alpha)\mu(d\alpha)=:g(t,x)-I(t,x).$$
Here, $k(g(t,x),\alpha)$ is a coefficient measuring the strength of response to the vector field $g$ at the point $x$, while the integral over $A$ sums up the total averaged dynamic friction.  
Motivated by such a physical intuition, one then can consider the controlled differential inclusion
\begin{equation}\label{eqn:dF}
  \dot{x}\in  g(t,x,u)-\int_A{k(t,x,u,\alpha)\partial_x\varphi(x,\alpha)\mu(d\alpha)},
\end{equation}
where now the control $u$ determines the choice of a vector field, the strength to the response $k$ also depends on the control and the measure $\mu$ is allowed to choose the relevant, averaging points over $S$ through $A$. $\varphi$ is a \textcolor{black}{possibly} non-smooth function and $\partial _x$ is a suitable sub-gradient
(precise definitions will be provided in the following sections). 

Differential equations with discontinuous right-hand side has been used for several tasks such as to model electric circuits, hysteresis phenomena and mechanical constraints (see, e.g. \textcolor{black}{\cite{acary2010nonsmooth,drabek1999nonlinear,tanwani2015observer, brogliato2020dynamical} }). More recently, discontinuous differential equations have been proposed to model the growth of stems and vines (see, e.g., \cite{bressan2017growth}, \cite{bressan2017well}) and we expect that  the dynamics \eqref{eqn:dF} can be also used for similar purposes (for instance, to model the evolution of a soft robotic device that moves in soil \textcolor{black}{\cite{del2017efficient}, \cite{del2016circumnutations}, \cite{del2019characterization}, \cite{tedone2020optimal}}). The dynamics  \eqref{eqn:dF} has some strong connections with the controlled perturbed sweeping process $\dot{x}\in g(t,x,u)-N_{C}(x)$, where $C$ is a closed set and $N_C(x)$ is a suitable normal cone to $C$ at $x$. \textcolor{black}{Indeed}, when the strength to the response $k$ in \eqref{eqn:dF} is sufficiently large and $\varphi(x,\alpha)=d(x,C(\alpha))$ (where $d(\cdot, C(\alpha))$ is the distance function from $C(\alpha)$), then the model \eqref{eqn:dF} can be regarded as a perturbed, ``averagely swept", sweeping process (see, e.g. \cite{thibault2003sweeping}, Theorem 3.2). Let us also mention that the averaging occurring in \eqref{eqn:dF} has a quite different character compared to the one presented in the  Riemann-Stiltjies control literature (see, e.g. \cite{ross2015riemann}, \cite{bettiol2019necessary}, \cite{palladino2016necessary}), since the averaging in \eqref{eqn:dF} occurs in the dynamics and not in the cost. 

In this paper, we will mainly concentrate on the basic properties of a control system driven by \eqref{eqn:dF}. Furthermore, we will derive the Hamilton-Jacobi equation for the related free time optimal control problem of Mayer type.  Results of a similar kind have been derived in \cite{ye1998hamilton}, \cite{ye2000discontinuous}, \cite{malisoff2002viscosity} in the case in which optimal stopping and optimal exit time are considered and the dynamics is Lipschitz continuous. It is important to mention that  the characterization of the value function of an optimal control problem with Lipschitz continuous dynamics $F$ relies on the strong invariance backward in time of the control system's state trajectories with respect to the epigraph of the value function. Such a property is a consequence of the tacit assumption \textit{``if $F$ is Lipschitz continuous, then $-F$ is also Lipschitz continuous"}.  An excellent overview on this approach is  provided in \cite{plaskacz2003value} and (\cite{vinter2010optimal}, Chapter 12).
However, since the control system \eqref{eqn:dF} is merely one-sided Lipschitz continuous, the strong invariance backward in time of the state trajectories does not hold. \textcolor{black}{It is interesting to observe that the non-validity of the strong invariance principle backward in time with respect to the epigraph of the value function is also closely related to some non-uniqueness phenomena which arise when the value function is merely lower semicontinuous (and, in general, the value function of a free-time optimal control problem with end-point constraint is merely lower semi-continuous). An example of this kind of issue can be found in \cite{serea2003reflecting}.}
Therefore, the characterization of the value function as the unique solution of a Hamilton-Jacobi equation does not follow from the standard theory and requires \textcolor{black}{a different approach} \cite{colombo2016minimum,donchev2005strong}. 

The paper is organised as follows: in Sections \ref{sec:prel}-\ref{sec:general} we will provide the basic concepts, the notations and the problem formulation that we will refer to throughout the whole paper. In Section \ref{sec:basic} we will study the well-posedness of the model as a control system; in Section \ref{sec:Existence} we will describe the properties of the related, free time, optimal control problem and of the associated value function. Sections \ref{sec:DPP}-\ref{sec:HJB} provide useful properties of the value function and its characterization as viscosity solution of a suitable Hamilton-Jacobi equation. In Section \ref{sec:ex} an example showing the effectiveness of the theory is provided. The proofs of some technical results, useful in the development of the theory, are provided in the Appendix.
\begin{figure}[t!]
\centering
\includegraphics[width=0.32\textwidth]{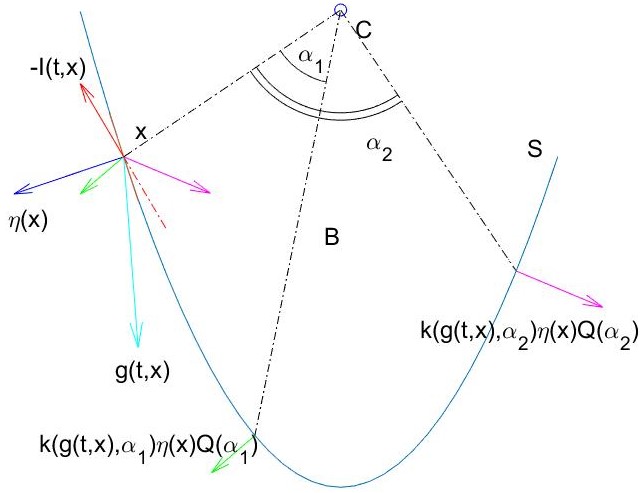}
\caption{\label{fig1} \textbf{An example generating the studied model}:\\ 
$\mu=\delta_{\alpha_1}+\delta_{\alpha_2}$ and $k(g,\alpha_i)=\lambda_i>0$ for $i=1,2$.
Therefore, the total friction at $x$ will be $-I(t,x)$, affecting $g(t,x)$.}
\end{figure}

\section{Preliminaries and Notations}\label{sec:prel}
\label{preliminaries}
In this section, we will recall some useful notations and concepts which will be used throughout the whole paper.
Let us use $\mathbb{B}$ to denote the open, unit ball and $\mathrm{Bd}(A)$ to denote the boundary of a set $A$. For a given closed set $C\subseteq \mathbb{R}^{n}$ and a point $x\in C$, the proximal normal cone to $C$ at $x$ is
\begin{equation}
\begin{array}{lr}
  N_C^P(x):=\left\{p\in\mathbb{R}^n: \exists \; M>0\; \mathrm{s.t.} \right . \\
 \left . \qquad  \qquad \qquad\qquad  \left<p,y-x \right>\leq M|y-x|^2\quad \forall y\in C \right\}
  \end{array}
\end{equation}
 For a given lower semi-continuous function $f:\mathbb{R}^{n}\rightarrow \mathbb{R}\cup\{\infty\}$, the domain of $f$ is $\mathrm{dom}(f):=\{x\in\mathbb{R}^n:\;f(x)<\infty \}$. The proximal sub-differential of $f$ at $x\in \mathrm{dom}(f)$ is
  \begin{displaymath}
  \partial_P f(x)=\left\{v\in\mathbb{R}^n:\;(v,-1)\in N_{\mathrm{epi}(f)}^P(x,f(x))\right\},
\end{displaymath}
where 
$\mathrm{epi}(f)= \left\{(x,\alpha)\in \mathbb{R}^{n+1}:\; f(x)\leq \alpha \right\}$ is the epigraph of the function $f$.
An equivalent characterization (see, e.g., \cite{vinter2010optimal}, Proposition $4.4.1$) of the proximal sub-differential is the following: $\xi \in \partial_{P} f(x)$ if there exist $M>0$ and $\varepsilon>0$ such that
\begin{equation}\label{eqn:prox-sub}
\left< \xi, y-x \right>\leq f(y) -f(x) + M|x-y|^{2},
\end{equation}
for each $y\in x+ \, \varepsilon\bar{\mathbb{B}}$. 
Furthermore, if $f:\mathbb{R}^{n}\rightarrow \mathbb{R}\cup\{\infty\}$ is locally Lipschitz continuous 
then  $\partial_{P} f(x)$ is locally bounded for each $x\in \mathrm{dom}(f)$. If $f:\mathbb{R}^{n}\rightarrow \mathbb{R}\cup\{\infty\}$ is lower semi-continuous and convex, then $\mathrm{epi}(f)$  is closed and convex. In particular this implies that $\partial_{P} f (x) \neq \emptyset$ for each $x\in\mathrm{dom}f$. Further, if $f$ is convex, then the proximal sub-differential $\partial_P f(x)$ coincides with the set
\begin{equation}\label{def:subdiff}
\begin{array}{lr}
\partial f(x):=\left\{\xi\in\mathbb{R}^n: f(z)\geq f(x)+\left<\xi,z-x\right> \right . \\[1mm]
 \left . \qquad  \qquad \qquad \qquad \qquad \qquad\forall z\in\mathrm{dom}(f)\right\}
  \end{array}
\end{equation}
and we will simply refer to it as subdifferential. 
It will be also helpful to define a notion of proximal super-differential. For a given upper semi-continuous function $f:\mathbb{R}^{n}\rightarrow \mathbb{R}\cup\{\infty\}$ and $x\in \mathrm{dom}\,(f)$, the proximal super-differential of $f$ at $x$ is
  \begin{displaymath}
  \partial^P f(x)=\left\{v\in\mathbb{R}^n:\;(-v,1)\in N_{\mathrm{hypo}(f)}^P(x,f(x))\right\},
\end{displaymath}
where
$\mathrm{hypo}(f)= \left\{(x,\alpha)\in \mathbb{R}^{n+1}:\; f(x)\geq \alpha \right\}$ is the hypograph of the function $f$.
Given $\Omega \subseteq \mathbb{R}^k$ and $M:\Omega\leadsto\mathbb{R}^r$, $M$ has closed graph if $\mathrm{Gr}\,M= \{(x,v):\; v\in M(x), \,x\in \Omega \}$ is closed. It is well known that if a multifunction is locally bounded and has closed graph, then $M$ is upper semicontinuous (see, a.g. \cite{van1999characteristic}, proposition AII.14). 
$M$ is said locally one-sided Lipschitz (OSL) if, for any compact $K\subset \Omega$, there is a constant $L\geq0$ such that $$\left<w-v, y-x\right>\leq L|x-y|^2,$$
for every $x,y \in K$, $v \in M(x)$ and $w \in M(y)$. 
Given a finite Radon measure $\mu$ and a  $\mu$-measurable set $A$, let us define
\begin{equation}
\begin{array}{lr}
L^{1}(A;\mu):=\left\{ g:A\rightarrow \mathbb{R}^n\;\mu\mathrm{-meas.}:\,\int_{A}\left|g(\alpha)\right|\mu(d\alpha) <\infty \right\}.
\end{array}
\end{equation}
Given a multifunction $\tilde\Gamma: \Omega\times A \leadsto\mathbb{R}^r$ , the parameterized integration  of $\tilde\Gamma$ (see, e.g., \cite{aumann1965integrals}, \cite{artstein1989parametrized}, \cite{saint1999parametrized}) is a new multifunction $\Gamma(x):=\int_A \tilde{\Gamma}(x,\alpha) \mu(d\alpha)$
where
\begin{equation}
\begin{array}{lr}
 \int_A \tilde{\Gamma}(x,\alpha) \mu(d\alpha):=\left\{\int_A \gamma(\alpha)\mu(d\alpha) :\; \gamma(\alpha)\right . \\ 
\left . \qquad \qquad\qquad \qquad \; \mu\mathrm{-measurable}\; \mathrm{selection} \; \mathrm{of}\; \tilde{\Gamma}(x,\cdot)\right\}
\end{array}.
\end{equation}

\section{The General Setting}\label{sec:general}
\label{intro}
Consider the optimal control problem \\
\begin{equation}\label{eqn:free}
(P)\;\begin{cases}
\mathrm{Minimize}\, W(T,x(T)) \\
\mathrm{over}\;\; T\geq t_0\;\; \mathrm{and}\;\; (x,u)\in AC([t_{0},T]; \mathbb{R}^{n})\times \mathcal{U}\;\;\mathrm{s.t.}\\
\dot{x}(t)\in F(t,x,u), \quad a.e.\ t\in[t_{0},T]\\
u(t)\in U\subset\mathbb{R}^m,\quad a.e.\ t\in[t_{0},T]\\
x(t_0)=x_0\in \mathbb{R}^n\\
(T,x(T))\in\mathrm{Gr}\,\T\subseteq\R^{1+n}
\end{cases}
\end{equation}
the data which comprise an initial time $t_0\in\R$, an initial state $x_0\in \mathbb{R}^{n}$, a cost function $W:\mathbb{R}^{1+n}\rightarrow \mathbb{R}$, a set $\mathcal{U}$ of measurable control functions $u$ defined on $[t_{0},+\infty)$ and taking values in a compact set $U\subset \mathbb{R}^m$, 
a controlled, non-empty multifunction $F:\mathbb{R}^{1+n}\times U\leadsto \mathbb{R}^n$ and a non-empty multifunction $\T:\R\leadsto \mathbb{R}^{n}$. We have used the symbol $AC([t_{0},T]; \mathbb{R}^{n})$ to denote the set of absolutely continuous functions from $[t_0, T]$ to $\mathbb{R}^n$. Notice that the multifunction $\T:\R\leadsto \mathbb{R}^{n}$ represents a moving target which has to be reached at the final time $T$. In particular, we will consider the case in which the controlled multifunction $F$ is defined as
\begin{equation}\label{eqn:F}
    F(t,x,u)=g(t,x,u)-\int_A{k(t,x,u,\alpha)\partial_x\varphi(x,\alpha)\mu(d\alpha)},
\end{equation}
where $A\subseteq\mathbb{R}^{\nu}$ is a given compact set, $k: \mathbb{R}^{1+n} \times U \times A \rightarrow \mathbb{R}^+$, $g:\mathbb{R}^{1+n}\times U\rightarrow \mathbb{R}^n$, $\varphi: \mathbb{R}^n\times A\rightarrow \mathbb{R}$ are given functions and $\mu$ is a finite Radon measure over $A$.  Sometimes, to emphasize the dependence on the initial condition, we will use $(P)_{(t_0,x_0)}$ to denote the optimal control problem $(P)$ with initial condition $x(t_0)=x_0$.
We shall assume the following standing assumptions \textbf{(SH)}:
\begin{itemize}
\item[$H_1$:] The maps $(t,x,u,\alpha)\mapsto k(t,x,u,\alpha)$, $(t,x,u)\mapsto g(t,x,u)$ and $(x,\alpha) \mapsto \varphi(x,\alpha)$  are  continuous.
\item[$H_2$:] For any compact sets $I\subset \mathbb{R}$, $K\subset \mathbb{R}^{n}$, there \textcolor{black}{exists a constant} $L>0$ such that
\begin{equation}
\begin{array}{llll}
\left| g(t,x,u) - g(s,y,u) \right| \leq L\left(|t-s|+|x -y| \right), \\[2mm]
\left| k(t,x,u,\alpha) - k(s,y,u,\alpha) \right| \leq L\left(|t-s|+|x -y| \right), \\[2mm]
\end{array}
\end{equation}
for every $(t,x)$, $(s,y)\in I\times K$, $u\in U$ and $\alpha \in A$.
\item[$H_3$:] There exist constants $C_1, C_2\geq0$ such that 
$$ 0\leq k(t,x,u,\alpha),  \, \left|g(t,x,u)\right| \leq C_1+C_2\left|x\right|,$$
for every  $(t,x,u,\alpha)\in \mathbb{R}\times \mathbb{R}^n\times U \times A$.
\item[$H_4$:] for each $\alpha\in A$, the mapping $x\mapsto \varphi(x,\alpha)$ is convex and globally Lipschitz continuous with constant $L_\varphi (\alpha)$, where $L_\varphi(\cdot)$ is a non-negative, $\mu$-integrable function.
\item[$H_5$:] the set-valued map $\bar{F}(t,x):=\cup_{u\in U}F(t,x,u)$ takes convex values for each $(t,x)\in  \mathbb{R}^{1+n}$.
\item[$H_6$:] the multifunction $\T:\R\leadsto \mathbb{R}^{n}$ has closed graph.
\item[$H_7$:] the function $W:\mathbb{R}^{1+n} \rightarrow \mathbb{R}$ is locally Lipschitz continuous in $\mathrm{Gr}\,\T+\e\B$, for some $\e>0$.
\end{itemize}

\section{Basic Properties of the Model}\label{sec:basic}

In this section, we will formally prove some important properties of the free time optimal control problem $(P)$. To this purpose, let us introduce the set-valued function
\begin{equation}\label{Mult_I}
    I(t,x,u)=\int_Ak(t,x,u,\alpha)\,\partial_x\varphi(x,\alpha)\mu(d\alpha).
\end{equation}
Let us also consider the set-valued map 
\begin{displaymath}
  \bar{F}(t,x)=\bigcup_{u\in U}\left\{g(t,x,u)-I(t,x,u)\right\} 
\end{displaymath}
for each $(t,x)\in \mathbb{R}^{1+n}$. The maps $F$ and $\bar{F}$ satisfy the following conditions:
\begin{proposition}\label{prop:properties}
\emph{Assume conditions $H_{1}$-$H_{4}$. Then the map $(t,x)\leadsto\bar{F}(t,x)$ is non-empty, compact and upper semi-continuous. Furthermore, for each $x\in \mathbb{R}^n$, the map $t\leadsto\bar{F}(t,x)$ is locally Lipschitz continuous and, for each $(t,u)\in\R\times U$, the map $x\to F(t,x,u)$ is locally OSL. In particular, for any compact set $I\times K \subset \mathbb{R}\times \mathbb{R}^n$ and for every $y_1=(t_1,x_1),\ y_2=(t_2,x_2)\in I\times K$, there exists a constant $L_{\bar{F}}$ such that
\begin{equation}\label{eqn:donchev_osl}
\begin{array}{lr}
\max_{v\in\bar{F}(y_1)}\left<v,x_1-x_2\right>-\\
\qquad \qquad -\max_{w\in\bar{F}(y_2)}\left<w,x_1-x_2\right> \leq L_{\bar{F}}|y_1-y_2|^2.
\end{array}
\end{equation}
}
\end{proposition}
\begin{proof}{\emph{Proof.}}
In view of the hypothesis $H_2$-$H_4$ on $\varphi$, one has that $\partial_x \varphi(x,\alpha)$ is non-empty, bounded by $L_\varphi(\alpha)\mathbb{B}$ and convex for each $(x,\alpha)\in \mathbb{R}^n \times A$ . The continuity of $\varphi$ with respect to $\alpha$ ensures that the graph of $\alpha\mapsto\partial_x \varphi(x,\alpha)$ is closed. Therefore, the map $\alpha \mapsto\partial_x \varphi(x,\alpha)$ admits a $\mu$-measurable selection  for each $x\in \mathbb{R}^n$  (see, e.g., Theorem 2.3.11 \cite{vinter2010optimal}) and
$\bar{F}(t,x)$ is non-empty for each $(t,x)\in \mathbb{R}^{1+n}$.  
Furthermore, $\bar{F}$ is locally bounded in view of $H_3$-$H_4$. 
Since $U$ is compact and in view of $H_1$-$H_2$, one can prove that the mapping $(t,x)\mapsto I(t,x,u)$ has closed graph for each $u\in U$. \textcolor{black}{Indeed}, fix $u\in U$, take $(t_k, x_k) \in \mathbb{R}^{1+n}$  converging to $(t,x)$ and $v_k \in I(t_k,x_k,u)$ for each $k\in \mathbb{N}$, converging to some $v$. We need to show that $v\in I(t,x,u)$. It follows from the definition of parameterized integration that
\begin{equation}
\begin{array}{lr}
I(t_k, x_k, u)=\left\{\int_A k(t_k,x_k,u,\alpha) \xi_k(\alpha) \mu(d\alpha) :\; \right . \\[2mm]
\left . \qquad\qquad\quad  \xi_k(\alpha)\; \mu\mathrm{-measurable}\; \mathrm{selection} \; \mathrm{of}\; \partial_x \varphi(x_k,\cdot)\right\}.
\end{array}
\end{equation}
Hence, \textcolor{black}{for any $\mu$-measurable sequence $\xi_k(\alpha)\in \partial_x \varphi(x_k,\alpha)$, there exists a subsequence (we do not relabel) which} weakly converges in $L^1(A;\mu)$ to a $\mu$-measurable selection $\xi(\alpha) \in \partial_x \varphi(x,\alpha)$ (see, e.g., Theorem 1, pg. 125, \cite{castaing2006convex}). Furthermore, in view of $H_2$, one easily obtains
\begin{equation}\label{eqn:K_Lipsch}
\max_{U\times A} \left|k(t_k,x_k,u,\alpha)-k(t,x,u,\alpha)\right|\leq L\left(|t_k-t|+|x_k-x|\right).
\end{equation}
Call
 $\epsilon_k= \int_Ak(t,x,u,\alpha)\left(\xi_k(\alpha)-\xi(\alpha)\right) \mu(d\alpha)$ 
and observe that, $\epsilon_k\rightarrow 0$ since  $\xi_k(\alpha)$ weakly converges in $L^1(A;\mu)$ to $\xi(\alpha)$.
In particular in view of \eqref{eqn:K_Lipsch}, one easily obtains
\begin{equation}\label{steps2}
\begin{array}{lll}
\int_A k(t_k,x_k,u,\alpha) \xi_k(\alpha) \mu(d\alpha)\in \int_A k(t,x,u,\alpha) \xi(\alpha)\mu(d\alpha)\\[2mm]
+\left(|\epsilon_k| + L\bigg(1+\int_A L_\varphi(\alpha)\mu(d\alpha)\bigg)\left(|t_k-t|+|x_k-x|\right)\right)\bar{\mathbb{B}} \\[2mm]
\qquad\qquad\qquad \qquad \qquad \qquad \qquad\qquad\subset I(t,x,u)+|\tilde\epsilon_k|\bar{\mathbb{B}}
\end{array}
\end{equation}
for some \textcolor{black}{$\tilde{\epsilon}_k\rightarrow 0$} when $k\rightarrow\infty$, which implies that $(t,x)\leadsto I(t,x,u) $ has closed graph for each $u\in U$. Since $I(t,x,u)$ is also locally bounded in view of $H_3$-$H_4$, one has that the map  $(t,x)\leadsto I(t,x,u) $ is upper semi-continuous for each $u\in U$. 
\textcolor{black}{Furthermore, it is then straightforward to} prove that $F(t,x,u)$ is upper semi-continuous for each $u\in U$. \\
\textcolor{black}{It follows from the upper-semicontinuous property of $(t,x)\leadsto F(t,x,u)$ for each $u\in U$ that also $(t,x)\leadsto \bar{F}(t,x)$ is upper-semicontinous. Indeed,} for each $u\in U$ fixed, for any $(t,x)\in \mathbb{R}^{1+n}$ and for every neighborhood $\mathcal{N}_u$ of $F(t,x,u)$, there exists a neighborhood $\mathcal{O}_u$ of $(t,x)$ such that  $F(s,y,u)\subset \mathcal{N}_u$ for any $(s,y)\in \mathcal{O}_u$. Let us now observe that $\mathcal{N}:=\cup_{u\in U} \mathcal{N}_u$ can be regarded as an open arbitrary neighborhood of $\bar{F}(t,x)$ and that $\bar{F}(s,y) \subseteq \mathcal{N}$, for every $(s,y)\in \mathcal{O}:=\cup_{u\in U} \mathcal{O}_u$. This shows that $\bar{F}$ is upper semi-continuous.
Furthermore, the local Lipschitz continuity of the map $t\leadsto \bar{F}(t,x)$ easily follows from the local Lipschitz continuity conditions expressed in $H_2$.\\ 
Let us now show that, for each $(t,u)\in\R\times U$, $F$ is locally one-sided Lipschitz w.r.t. $x\in \mathbb{R}^n$. Fix any compact sets $I\subset \mathbb{R}$, $K\subset  \mathbb{R}^{n}$ and any $(t,x,u),(t,y,u)\in I\times K\times U$.
For every $v\in F(t,x,u)$, $w\in F(t,y,u)$ there exist measurable selections $\eta_x(\alpha)\in\partial_x\varphi(x,\alpha)$ and $\eta_y(\alpha)\in\partial_y\varphi(y,\alpha)$, $\mu$-a.a. $\alpha\in A$, such that
\begin{align*}
    &v=g(t,x,u)-\int_A{k(t,x,\alpha,u)\eta_x(\alpha)\mu(d\alpha)},\\
    &w=g(t,y,u)-\int_A{k(t,y,\alpha,u)\eta_y(\alpha)\mu(d\alpha)}.
\end{align*}
Therefore, one can derive the following inequalities:
\begin{equation}\label{OSL}
\begin{array}{rr}
    \left<x-y,v-w\right>\leq L|x-y|^2+&\\[2mm]
    \int_A{k(t,y,u,\alpha)\left<\eta_y(\alpha),x-y\right>\mu(d\alpha)}+&\\[2mm]
    \int_A{k(t,x,u,\alpha)\left<\eta_x(\alpha),y-x\right>\mu(d\alpha)}\leq&\\[2mm]
    L|x-y|^2+\int_Ak(t,y,u,\alpha)(\varphi(x,\alpha)-\varphi(y,\alpha))\mu(d\alpha)-\\[2mm]
    \int_Ak(t,x,u,\alpha)(\varphi(x,\alpha)-\varphi(y,\alpha))\mu(d\alpha)\leq L|x-y|^2+&\\[2mm]
     +|x-y|\int_A{L_\varphi(\alpha)|k(t,y,u,\alpha)-k(t,x,u,\alpha)|\mu(d\alpha)}&\\[2mm]
    \leq L\bigg(1+\int_A L_\varphi(\alpha)\mu(d\alpha)\bigg)|x-y|^2=L_F|x-y|^2,
\end{array}
\end{equation}
for each $x,y\in K$, $t\in I$, $u\in U$, where, in turns, we have used the characterization \eqref{def:subdiff} of the proximal sub-differential, hypotheses $H_2$, $H_4$ and the positivity of $k$. This shows that, for each $t\in \mathbb{R}$, $u\in U$, $F$ is locally OSL w.r.t. $x$.\\
In order to prove \eqref{eqn:donchev_osl}, fix any  $(t_1,x_1), (t_2, x_2) \in I\times K$. Let  $u\in U$ and  $v_1\in F(t_1,x_1,u)$ be such that 
$$\max_{v\in \bar{F}(t_1, x_1)}\left< v, x_1-x_2\right>=\left<v_1, x_1-x_2\right>.$$
 Fix any $v_2\in F(t_2,x_2,u)$ and choose $w\in F(t_1,x_2,u)$ such that 
$$|w-v_2|\leq L|t_1-t_2|. $$
Then one can easily estimate 
\begin{align*}
    &\max_{v\in\bar{F}(t_1, x_1)}\left< v, x_1-x_2\right>-\max_{w\in\bar{F}(t_2, x_2)}\left< w, x_1-x_2\right>\leq\\
    &\left<v_1,x_1-x_2\right>-\left<v_2,x_1-x_2\right>=\\
    &\left<v_1-w,x_1-x_2\right>+\left<w-v_2,x_1-x_2\right>\leq\\
    &L_F|x_1-x_2|^2+|w-v_2||x_1-x_2|\leq\\ &L_F|x_1-x_2|^2+L|t_1-t_2||x_1-x_2|\leq\\
    &(L_F+L)|(t_1,x_1)-(t_2,x_2)|^2=
    L_{\bar{F}}|(t_1,x_1)-(t_2,x_2)|^2,
\end{align*}
where $L_{\bar{F}}=(L_F+L)$. This shows relation \eqref{eqn:donchev_osl} and concludes the proof.
\end{proof}

Let us now consider the control system
\begin{equation}\label{eqn:dynamics}
\begin{cases}
\dot{x}(t)\in F(t,x(t),u(t)), \; u\in \mathcal{U},\; \mathrm{a.e.}\ t\in[t_{0},+\infty)\\[1mm]
x(t_0)=x_0\in \mathbb{R}^n
\end{cases} .
\end{equation}
\begin{remark}\label{rem:traj}
Notice that, as a consequence of the one-sided Lipschitz property \eqref{eqn:donchev_osl}, for every $T>t_2\geq t_1$, for every $x_1, x_2$ such that $|x_1|,|x_2|\leq r$ and for every solution of \eqref{eqn:dynamics} $x_1(\cdot),\ x_2(\cdot)$, respectively starting from \textcolor{black}{$x_1(t_1)=x_1$, $ x_2(t_2)=x_2$} with a given control $u\in \mathcal{U}$, one has
\begin{equation}\label{Lipschitz}
|x_1(t)-x_2(t)|\leq e^{L_{\bar F}(t-t_2)}|x_1(t_2)-x_2(t_2)|,
\end{equation}
 for all $t\in [t_2, T]$. Let us also observe that $x_i(\cdot)$ for $i=1,2$, satisfies the bound
 \begin{equation}\label{Gronwall_Hyp}
 \left|x_i(t)-x_i\right|\leq \left(1+\int_A L_\varphi(\alpha)\mu(d\alpha) \right)\int_{t_i}^t \left(C_1+C_2 |x_i(s)|\right)ds 
\end{equation}
where $C_1, C_2$ and $L_\varphi(\alpha)$ are defined in $H_3$-$H_4$, respectively.
It follows from \eqref{Lipschitz}  and \eqref{Gronwall_Hyp} that
\begin{equation}\label{Lipschitz_2}
\begin{array}{lll}
|x_1(t)-x_2(t)|\leq e^{L_{\bar F}(t-t_2)}\left(|x_1(t_2)-x_1|+ |x_1-x_2| \right) \\[2mm]
\leq e^{L_{\bar F}(t-t_2)}\left( C_r(T)|t_1-t_2|+|x_1-x_2| \right)\\
\leq \lambda_{r}(t)|(t_1,x_1)-(t_2,x_2)|,
\end{array}
\end{equation}
where $C_r(t):=[(L_{F}/L)(C_1+C_2r)]\mathrm{exp}\left\{ (L_{F}/L)(t-t_1)\right\}$ is obtained from \eqref{Gronwall_Hyp} and a use of the \textcolor{black}{Gr\"onwall's} Lemma, $L>0$ is the constant appearing in $H_2$ and $\lambda_r(t):=2e^{L_{\bar{F}}t}\max\left\{C_r(t),1\right\}>1$.
\end{remark}

An important consequence of Proposition \ref{prop:properties} is that the control system \eqref{eqn:dynamics}
is well-posed, as it is stated in the following result. 
\begin{theorem}\label{prop:exist_unique}
\emph{ Assume the hypotheses $H_1$-$H_5$. For a given $ (t_0,x_0)\in \mathbb{R}^{1+n}$ and $u\in \mathcal{U}$, there exists a unique solution to \eqref{eqn:dynamics}}.
\end{theorem}
\begin{proof}{\emph{Proof.}}
The existence of a solution follows from the properties stated in Proposition \ref{prop:properties} and the use of well-known existence results for differential inclusions (see, e.g. \cite{lojasiewicz1985some}). Moreover,  the uniqueness property for the system \eqref{eqn:dynamics} follows from \eqref{Lipschitz_2} and an application of the \textcolor{black}{Gr\"onwall's} Lemma.
\end{proof}

Furthermore, one can show that the set of trajectories generated by the dynamics \eqref{eqn:dynamics} is equivalent to the set of solutions of
\begin{equation}\label{eqn:bar_F}
\begin{cases}
\dot{x}(t)\in \bar{F}(t,x(t)),\quad \mathrm{a.e.}\ t\in[t_{0},+\infty)\\[1mm]
x(t_0)=x_0\in \mathbb{R}^n
\end{cases}.
\end{equation}
One has the following result:
\begin{proposition}\label{prop:equivalence}
\emph{Let us assume $H_1$-$H_5$. Fix $(t_0,x_0)\in \mathbb{R}^{1+n}$. Then the set of solutions of \eqref{eqn:dynamics} with initial condition $x(t_0)=x_0$ is equal to the set of solutions of \eqref{eqn:bar_F} with initial condition $x(t_0)=x_0$.}
\end{proposition}
\begin{proof}{\emph{Proof.}} If $x(\cdot)$ is a solution of \eqref{eqn:dynamics} with initial condition $x(t_0)=x_0$, then it is trivially also a solution of  \eqref{eqn:bar_F} with the same initial condition. Let us now take $x(\cdot)$ solution of \eqref{eqn:bar_F} such that $x(t_0)=x_0$. In what follows, we will equip $L^{1}(A; \mu)$ 
with its natural weak topology. Let us consider the multifunction $\Sigma:[t_0,\infty)\leadsto L^1(A; \mu)\times U$ defined as
\[
\begin{array}{lr}
\Sigma(t):=\left\{(\xi(\cdot),u):\; u\in U,\; \xi(\cdot)\right .\\ 
\left . \qquad \qquad\qquad\; \mu\mathrm{-measurable}\; \mathrm{selection} \; \mathrm{of}\; \partial_{x}\varphi(x(t),\cdot)\right\}
\end{array}
\]
and the mapping $\tilde{g}:[t_0,\infty)\times U\times L^{1}(A; \mu)\rightarrow \mathbb{R}^n$ defined as
$$\tilde{g}(t,u,\xi):=g(t,x(t),u)-\int_A k(t,x(t),u,\alpha)\xi(\alpha)\mu(d\alpha).$$
It is a straightforward matter to check that $\Sigma$ is non-empty (in view of $H_4$) and has weakly closed graph (in view of the compactness of $U$ and of the upper-semicontinuity and convexity of the sub-differential). Furthermore, in view of $H_1, H_2$, the map $\tilde{g}$ is weakly continuous. Notice also that the relation
$$\dot{x}(t)\in \{\tilde{g}(t,u,\xi):\quad (u,\xi)\in \Sigma(t) \},\qquad \mathrm{a.e.}\;t\in[t_0,\infty),$$
is clearly satisfied.
 So one can apply a well-known selection theorem (see, e.g. Theorem III.38, \cite{castaing2006convex}), which provides the existence of a measurable selection $(\xi(t), u(t))\in \Sigma(t)$ such that $\dot{x}(t)=\tilde{g}(t,u(t),\xi(t))$, a.e. $t\in[t_0,\infty)$. This concludes the proof.
\end{proof}
\section{Existence of Minimizers and Properties of the Value Function}\label{sec:Existence}

Fix $(t_0,x_0)\in \R^{1+n}$. Let us now define the reachable set generated by the dynamics \eqref{eqn:bar_F} and starting from the point $x(t_0)=x_0$, evaluated at $s\geq t_0$ (in view of Proposition \ref{prop:equivalence}, one can regard any trajectory of \eqref{eqn:bar_F} as a trajectory of \eqref{eqn:dynamics} and vice-versa):
$$R(s;t_0,x_0)=\left\{x(s):\, \dot{x}\in\bar{F}(t,x),\ t\in[t_0,s],\ x(t_0)=x_0\right\}.$$
The set of points of $\mathrm{Gr}\,\mathcal{T}$ reached by a trajectory of \eqref{eqn:dynamics} starting from $x(t_0)=x_0$ is defined as  
$$\mathcal{A}_{(t_0,x_0)}=\left\{(s,y)\in \mathrm{Gr}\,\T:\, y\in R(s;t_0,x_0),\ s\geq t_0\right\},$$
while the set of initial conditions for which a feasible trajectory exists is denoted by 
$$\mathcal{D}=\left\{(t_0,x_0)\in\R^{1+n}:\, \mathcal{A}_{(t_0,x_0)}\neq\emptyset\right\}.$$
In order to guarantee the existence of a minimizer, one has to assume further conditions, characterizing the behaviour  of the cost function $W$ when the end-time $T>t_0$ tends to infinity.  
In what follows, we will assume the following growth condition:
\begin{itemize}
\item[\textbf{(GC)}]
Fix $(t_0,x_0)\in \R^{1+n}$. For every $(T_k,x_k)\in \A_{(t_0,x_0)}$ such that $T_k\rightarrow +\infty$, one has that $W(T_k,x_k)\rightarrow +\infty$.
\end{itemize}
Clearly, if $W$ is a function such that $\sup_{x\in\mathbb{R}^n }W(T,x) \rightarrow \infty$ if $T\rightarrow \infty$, then the condition \textbf{(GC)} is satisfied. Let us point out that, in the minimum time problem, the cost function is $W(T,x)=T$ and  it clearly  satisfies the growth condition \textbf{(GC)}.
We are now ready to prove the existence of a minimizer for the optimal control problem $(P)$:
\begin{theorem}\label{prop:free}
\emph{Assume hypothesis \textbf{(SH)} and that condition \textbf{(GC)} is satisfied. Then, for any $(t_0,x_0)\in\D$, there exists a minimizer for the free time optimal control problem $(P)$}.
\end{theorem}
\begin{proof}{\emph{Proof.}}
Fix $(t_0,x_0)\in \D$.  Let $(T_n,x_n)_{n\in\mathbb{N}}$ be a minimising sequence in $ \A_{(t_0,x_0)}$. In particular, $(T_n,x_n)\in \mathrm{Gr}\,\T$ and $T_n$ has to be bounded. \textcolor{black}{Indeed}, if $T_n$ were not bounded, there would exist a subsequence such that $W(T_n,x_{n}(T_n))\rightarrow +\infty$, providing a contradiction with the definition of minimising sequence.  Let $M>0$ be such that $T_n\leq M$ for each $n$. In view of conditions $H_3$, $H_4$ and $H_5$, it follows from standard compactness arguments (see Proposition 2.5.3,  \cite{vinter2010optimal}) that $T_n\rightarrow T^*$ and $x_n(\cdot)\rightarrow x^{*}(\cdot)$ uniformly on $[t_0,M]$ (here, we are considering trajectories $x_n$, $x^*$ extended on $[t_0,M]$  such that $x_{n}(t)=x_n(T_n)$ for $T_n\leq t \leq M$ and $x^{*}(t)=x^{*}(T^*)$ for $T^*\leq t \leq M$). It follows from Proposition \ref{prop:properties} and assumptions $H_5$-$H_6$ that the set $R(s; t_0,x_0)$ is compact for every $t_0\leq s\leq M$ (see, e.g., Proposition $2.6.1$, \cite{vinter2010optimal}). Since $W$ is continuous \textcolor{black}{on} $\mathrm{Gr}\mathcal{T}$, this concludes the proof.
\end{proof}

Let us now introduce, for all $(t_0,x_0)\in \mathbb{R}^{1+n}$, the value function of the free time optimal control problem $(P)$ as
\begin{equation}\label{eqn:value-function}
V(t_0,x_0)=\inf\left\{W(T,x):\;(T,x)\in\mathcal{A}_{(t_0,x_0)}\right\}.
\end{equation}
Notice that $V(t_0,x_0)=\infty$ if $(t_0,x_0)\notin \mathcal{D}$. The standard dynamic programming principle for the optimal control problem $(P)$ can be stated as follows:
\begin{proposition}\label{prop:principle}
\emph{For any $(t,x)\in\mathcal{D}$, take  $y:[t,+\infty)\to\mathbb{R}^n$ such that $y(t)=x$ solution of \eqref{eqn:dynamics} with a control $u\in \mathcal{U}$. 
Then, for any $s\in [t,\infty)$ the value function satisfies}
\begin{displaymath}
  V(t,x)\leq V(s,y(s)).
\end{displaymath}
\emph{Furthermore, consider $\bar{y}:[t,+\infty)\to\mathbb{R}^n$ such that $(\bar{T}, \bar{y}(\cdot))$ is a minimizer for $(P)_{(t,x)}$. Then for any $t\leq s\leq \bar{T}$, one has} 
$$V(t,x)=V(s,\bar{y}(s)).$$
\end{proposition}


If the growth condition \textbf{(GC)} on $W$ is satisfied, one can easily derive also a related growth condition on the value function.

\begin{proposition}\label{prop:gc}
\emph{Assume \textbf{(SH)} and that condition \textbf{(GC)} is satisfied. Then the following growth condition holds:
\begin{itemize}
\item[$\;(\textbf{GC})_V$] For every $(t_k, x_k)\in \D$ such that $t_k\rightarrow \infty$, one has\\ that $V(t_k,x_k)\rightarrow \infty$.
\end{itemize}}
\end{proposition}
\begin{proof}{\emph{Proof.}}
Take $(t_k,x_k)\in \D$ such that $t_k\rightarrow \infty$. It follows from the definition of value function that for each $\varepsilon_k>0$, there exists $(T_{\varepsilon_k}, y_{\varepsilon_k})\in \A_{(t_k, x_k)}$ such that
\begin{equation}
W(T_{\varepsilon_k}, y_{\varepsilon_k})\leq V(t_k, x_k)+\varepsilon_k.
\end{equation}
Let us assume that $\varepsilon_k\rightarrow 0$. Since $T_{\varepsilon_k}\geq t_k$, then also $T_{\varepsilon_k}\rightarrow\infty$ for $k\rightarrow \infty$. It follows from the condition \textbf{(GC)} on $W$ that
\begin{equation}
\infty=\lim_{k\rightarrow \infty} W(t_{\varepsilon_k}, y_{\varepsilon_k})\leq \lim_{k\rightarrow \infty} V(t_k, x_k).
\end{equation}
This concludes the proof.
\end{proof}

In order to derive the Hamilton-Jacobi equation for the problem $(P)$, it will be helpful to impose conditions which guarantee the locally Lipschitz continuous regularity in $\mathcal{D}$ of the value function.
To this aim, we will extend to the one-sided Lipschitz case some results provided in \cite{veliov1997lipschitz}. Let us assume the following inward pointing condition on $\mathrm{Gr}\,\T$:
\begin{itemize}
    \item [\textbf{(IPC)}] For any compact set $G\subseteq\R^{1+n}$ there exists $\rho>0$ such that, for all $(t,x)\in\mathrm{Bd}\left(\mathrm{Gr}\,\T\right)\cap G$,
    \begin{align*}
        &\min_{\xi\in\bar{F}(t,x)}\left\{l^0+\left<l,\xi\right>\right\}\leq-\rho,\\
        &\qquad\qquad\qquad\quad\forall\ (l^0,l)\in N_{\mathrm{Gr}\T}^P(t,x),\ |(l^0,l)|=1.
    \end{align*}
\end{itemize}
It is then possible to prove the following technical result:
\begin{proposition}\label{prop:claim}
Assume $H_1$-$H_6$ hold and $\mathrm{Gr}\,\T$ satisfies \textbf{(IPC)}. Then, for any compact set $K\subseteq\R^{1+n}$, there exist $\varepsilon_K,\ L_K>0$ such that for all $(t_0,x_0)\in\mathrm{Gr}\,\T\cap K+\varepsilon_K\bar{\B}$
\begin{equation}\label{eqn:claim}
d((t_0,x_0),\A_{(t_0,x_0)})\leq L_Kd((t_0,x_0),\mathrm{Gr}\,\T)
\end{equation}
\end{proposition}
\begin{proof}{\emph{Proof.}}
See Section \ref{proof:claim}.
\end{proof}

\begin{proposition}\label{prop:Lips}
\emph{Assume conditions \textbf{(SH)} and \textbf{(GC)}. Suppose that $\mathrm{Gr}\,\T$ satisfies \textbf{(IPC)}. 
Then $\D\subseteq \mathbb{R}^{1+n}$ is an open set and $V(t_k,x_k)\rightarrow+\infty$ for all $(t_k,x_k)\in\D$ such that $(t_k,x_k)\rightarrow(t_0,x_0)\in\mathrm{Bd}(\D)$.}
\end{proposition}
\begin{proof}{\emph{Proof.}}
Let us show that $\mathcal{D}^{\mathsf{c}}$, the complement of $\mathcal{D}$, is closed.
Let $(t_n,x_n)$ be a sequence in $\mathcal{D}^{\mathsf{c}}$ converging to $(t,x)$. Fix $R>0$ such that $|(t,x)|<R$. We will show that $(t,x)\in\mathcal{D}^{\mathsf{c}}$.  

\textcolor{black}{Let us argue by contradiction assuming} that $(t,x)\in\mathcal{D}$. By definition of $\mathcal{D}$, there exist $(T,y)\in\mathcal{A}_{(t,x)}$, a control $u_x\in \mathcal{U}$ and $x(\cdot)$ solution of \eqref{eqn:dynamics} with control  $u_x\in \mathcal{U}$ such that $x(t)=x,\ x(T)=y,\ T\geq t$ and $(T,y)\in \mathrm{Gr}\,\T$.
\textcolor{black}{Clearly, one can choose $\e_n>0$ such that $\e_n\rightarrow 0$,} $$|(t,x)-(t_n,x_n)|<\e_n\textcolor{black}{,}$$
and $|(t,x)|+\e_n <R$. Let us consider two different cases:

\textbf{CASE 1: $T=t$.} Let $K\subseteq\R^{1+n}$ be a compact set containing $(t,x)=(T,y)\in \mathrm{Gr}\,\T$. Then
 $$d((t_n,x_n),\mathrm{Gr}\T)<|(t_n,x_n)-(T,y)|=|(t_n,x_n)-(t,x)|<\e_n$$ 
 and $(t_n,x_n)\in\mathrm{Gr}\,\T\cap K+\e_n\bar{\B}$. 
 In view of Proposition \ref{prop:claim}, there exist $\e_K,\ L_K>0$ \textcolor{black}{such that condition \eqref{eqn:claim} is satisfied for each point in $\mathrm{Gr}\,\T\cap K+\varepsilon_K\bar{\B}$. By taking $n$ sufficiently large, one can choose $\e_n$ such that $\e_n\leq\e_K$ and apply condition  \eqref{eqn:claim} at $(t_n, x_n)$. Since $(t_n,x_n)\not\in\mathcal{D}$, then $\mathcal{A}_{(t_n,x_n)}=\emptyset$, implying that $d((t_n,x_n),\mathcal{A}_{(t_n,x_n)})=+\infty$, which yields a contradiction.}

\textbf{CASE 2: $T>t$.}
Taking $n$ sufficiently large,  one can assume that $t_n<T$. Let $x_n(\cdot)$ be the unique trajectory of \eqref{eqn:dynamics} when $u(\cdot)=u_x(\cdot)$ with initial condition $x(t_n)=x_n$.
In view of relation \eqref{Lipschitz_2}, one has that
\begin{equation}
|x_n(T)-x(T)|<\lambda_{R}(T)\varepsilon_n
\end{equation}
where $\lambda_{R}(\cdot)$ is the function appearing in condition \eqref{Lipschitz_2}, Remark \ref{rem:traj} with the choice $r=R$. \textcolor{black}{Clearly, since $(t_n,x_n)\in \mathcal{D}^c$, then also the point $(T,x_n(T))$ is in  $\mathcal{D}^c$.}
Fix \textcolor{black}{a compact set} $K\subseteq\R^{1+n}$  containing $(T,y)=(T,x(T))$. Let us choose $\e_K,\ L_K$ such that the statement of Proposition \ref{prop:claim} is satisfied. 
For all $\e_n<\frac{\e_K}{\lambda_{R}(T)}$ it follows that $(T,x_n(T))\in \mathrm{Gr}\,\T\cap K+\e_K\bar{\B}$ \textcolor{black}{and that $$d((T,x_n(T)),\A_{(T,x_n(T))})\leq L_Kd((T,x_n(T)),\mathrm{Gr}\,\T).$$}\textcolor{black}{However, since $(T,x_n(T))\not\in\mathcal{D}$, one obtains a contradiction using the same arguments employed in \textbf{CASE 1}.}

\textcolor{black}{In both \textbf{CASE 1} and \textbf{CASE 2}, a contradiction is obtained. This shows} that $\mathcal{D}^c$ is closed and so that $\mathcal{D}$ is an open set.

Let us now prove that the value function $V(t,x)$ tends to infinity when $(t,x)$ approaches $\mathrm{Bd}(\D)$. Fix $\{(t_k,x_k)\}_k\subseteq\D$, $(t_k,x_k)\rightarrow(t_0,x_0)\in\mathrm{Bd}(\D)$. Let us use $x_k(\cdot)$ to denote a trajectory of  \eqref{eqn:bar_F} with initial condition $x_k(t_k)=x_k$. 

Assume by contradiction that $|V(t_k,x_k)|<\bar{M}$ for all $k$ and some $\bar{M}>0$. It follows from the definition of value function that, for all $\e_k>0$, there exists $(T_k,y_k)\in\A_{(t_k,x_k)}$ such that
 $$W(T_k,y_k)\leq V(t_k,x_k)+\e_k.$$ 
 Hence, $|W(T_k,y_k)|<+\infty$. Let us observe that, in view of \textbf{(GC)}, $T_k$ has to be bounded by a constant $M$. Hence, in view of the hypothesis $H_3$-$H_4$, also $y_k$ is uniformly bounded and one can arrange along a subsequence (we do not relabel) that  $(T_k,y_k)\rightarrow(T_0,y_0)\in\mathrm{Gr}\,\T$. Arguing as in Theorem \ref{prop:free}, one can find a subsequence of trajectories of \eqref{eqn:bar_F} such that  $x_k(\cdot)\rightarrow \tilde{x}(\cdot)$ uniformly on $[0,M]$. In particular, $x_k=x_k(t_k)\rightarrow \tilde{x}(t_0)=x_0$. Hence $\tilde{x}(\cdot)$ is a trajectory starting from $(t_0,x_0)$, such that $\dot{\tilde{x}}(t)\in \bar{F}(t,\tilde{x}(t))$ a.e. $t\in[t_0,T_0]$ and $(T_0,y_0)=(T_0,\tilde{x}(T_0))\in\mathrm{Gr}\,\T$. Then $\A_{(t_0,x_0)}\neq\emptyset$, which is impossible since $\D$ is open and $(t_0,x_0)\in\mathrm{Bd}(\D)$. This concludes the proof.
\end{proof}

The existence of a minimizer, together with \textcolor{black}{the} \textbf{(IPC)} and \textbf{(GC)} conditions guarantee the locally Lipschitz continuity of the value function on $\D$.
\begin{theorem}\label{prop:value_lips}
\emph{Assume that conditions \textbf{(SH)}, \textbf{(GC)} hold and that $\mathrm{Gr}\,\T$ satisfies \textbf{(IPC)}. Then $V$ is locally Lipschitz on $\D$.}
\end{theorem}
\begin{proof}{\emph{Proof.}}
See Section \ref{proof:claim}.
\end{proof}
\begin{remark}
The growth condition \textbf{(GC)} permits to the optimal trajectory $\bar{x}(\cdot)$ of problem $(P)_{(t_0,x_0)}$ to reach the point in $\mathcal{A}_{(t_0,x_0)}$ which minimizes the cost function $W$. In general, $\bar{x}(\cdot)$ \textit{does not} stop when the target is reached, as it is the case in which one considers a problem $(P)$ in which the parameter to minimize is the time (namely, when $W(T,x)=T$). Indeed, the related cost function in the minimum time problem  satisfies the stronger condition:
\begin{itemize}
    \item \textbf{(LGC)}. For any $K\subseteq\R^{n+1}$ compact, there exists $\gamma>0$ such that
    \begin{equation*}
        W(t',x')\geq W(t,x)+\gamma(t'-t),
    \end{equation*}
    for all $(t,x)\in K$ and $(t',x')\in\A_{(t,x)}$.
\end{itemize}
This particular feature of the problem of study is also reflected in the formulation of the Hamilton-Jacobi equation.
\end{remark}
\section{Dynamic Programming and Invariance Principles}\label{sec:DPP}
\label{invariance}
In this section, we link the dynamic programming principle in Proposition \ref{prop:principle} with the weak and strong invariance principles for the epigraph and the hypograph of the value function w.r.t. a suitable, augmented dynamics. To this aim, let us now introduce the augmented differential inclusion
\begin{equation}\label{eqn:gamma}
    (AD)_{y_0}\,\left\{\begin{array}{lll}
    &\dot{y}(t)\in\Gamma(y(t)),\quad a.e.\;t\in[0,+\infty)\\[2mm]
    &y(0)=y_0
    \end{array}\right . 
\end{equation}
where $y(t)=(\tau(t),x(t),a(t))$, $y_0=(\tau_0, x_0, a_0)\in \R^{1+n+1}$ and $\Gamma(\tau,x,a)=\{1\}\times \bar{F}(\tau,x)\times\{0\}$. It is easy to check that all of the properties stated in Proposition \ref{prop:properties} for $\bar{F}$ are still valid for $\Gamma$.
\begin{definition}\label{def:escape_f}
Suppose $\mathcal{O}\subseteq \mathbb{R}^{1+n+1}$ is open, $y_0\in \mathcal{O}$ and $y(\cdot)$ solution of $(AD)_{y_0}$. Then $T> 0$ is an escape time from $\mathcal{O}$ (in which case we write $\mathrm{Esc}(y(\cdot),\mathcal{O}) := T$), provided at least one of the following conditions occurs:
\begin{itemize}
\item[a)] $T=\infty$ and $y(t)\in \mathcal{O}$ for all $t\geq 0$;
\item[b)] $y(t)\in \mathcal{O}$ for all $t\in [0,T)$ and $||y(t)||\rightarrow \infty$ as $t\rightarrow T;$
\item[c)] $T<\infty$, $y(t)\in \mathcal{O}$ for all $t\in [0,T)$ and $d(y(t), \mathcal{O}^{c})\rightarrow 0$ as $t\rightarrow T$.
\end{itemize}
\end{definition}
Let us recall the basic definitions of invariance principles (\cite{wolenski1998proximal}, Definition $3.1$). In particular, we will state a local version of the weak invariance principle which will be useful in proving the following results. 
\begin{definition}\label{def:weak}
 \emph{Take a closed set $\mathcal{C}\subseteq \mathbb{R}^{1+n+1}$ and an open set $\mathcal{O}\subseteq \mathbb{R}^{1+n+1}$.  $\mathcal{C}$ is weakly invariant w.r.t. the set-valued dynamics $\Gamma$ in $\mathcal{O}$ (and we write $(\mathcal{C},\Gamma)$ weakly invariant in $\mathcal{O}$) if and only if, for any initial condition $y_0\in \mathcal{C}\cap \mathcal{O}$ and for some $T>0$, the Cauchy problem $(AD)_{y_0}$
admits a solution $y(t)\in \mathcal{C}\cap \mathcal{O}$ for all $t\in[0,T)$.}
\end{definition}
\begin{definition}\label{def:strong}
\emph{ A closed set $\mathcal{C}\subseteq \mathbb{R}^{1+n+1}$ is strongly invariant w.r.t. the set-valued dynamics $\Gamma$ (and we write $(\mathcal{C},\Gamma)$ strongly invariant) if and only if, for any $y_0\in \mathcal{C}$, $T\geq 0$ and $y:[0,T]\to\mathbb{R}^{1+n+1}$ solution of $(AD)_{y_0}$, one has $y(t)\in \mathcal{C}$ for all $t\in[0,T]$}.
\end{definition}
The existence of an optimal trajectory can be reformulated  as both a weak invariance principle for the epigraph of $V$ and a  strong invariance principle for the hypograph of $V$. Such properties will be captured by the next propositions.\\
\begin{remark}\label{rem:augmented}
Fix $y_0=(\tau_0,x_0,a_0)\in\mathbb{R}^{1+n+1}$.  Any solution $y(t)=(\tau(t),x(t),a(t))$ of $(AD)_{y_0}$ is such that $\tau(t)=t+\tau_0$. The inverse  function of $\tau(t)$ is $t(\tau)=\tau-\tau_0$. Furthermore, one can observe that $z(\tau):=x(t(\tau))$ satisfies $\dot{z}(\tau)\in\bar{F}(\tau,z(\tau))$ a.e. $\tau\in[\tau_0,\infty)$ and $z(\tau_0)=x_0$.
\end{remark}
\begin{proposition}\label{prop:weak}
\emph{Assume that \textbf{(SH)}, \textbf{(IPC)} and \textbf{(GC)} hold true and $V$ is bounded below and lower semi-continuous. Fix $\mathcal{E}=\mathrm{epi}(V)$, where}
\begin{displaymath}
\mathrm{epi}(V)=\{(\tau,x,\beta)\in\mathbb{R}^{1+n+1}:\,V(\tau,x)\leq\beta\},
\end{displaymath}  
\emph{and the set $\mathcal{O}=\left(\Omega^{\mathsf{c}}\cap \D \right)\times\R$, where $$\Omega=\{(\tau,x)\in \mathrm{Gr}\,\T:\; V(\tau,x)=W(\tau,x) \}.$$
Then $(\mathcal{E},\Gamma)$ is weakly invariant in $\mathcal{O}$.}
\end{proposition}
\begin{proof}{\emph{Proof.}}
Since $V$ is lower semi-continuous, $\mathcal{E}$ is a closed set. Furthermore, $\mathcal{O}$ is an open set in view of Proposition \ref{prop:gc}. Fix $(\tau_0,x_0,\beta_0)\in \mathcal{E} \cap \mathcal{O}$. Then $V(\tau_0,x_0)\leq \beta_0<\infty$. Theorem \ref{prop:free} ensures the existence of an optimal solution $z^*(\tau)$ and an optimal time $S^*> \tau_0$ to the free time optimal control problem $(P)$ with initial condition $z^*(\tau_0)=x_0$ and such that $(S^*,z^*(S^*))\in\mathrm{Gr}\T$ (here, in view of Proposition \ref{prop:equivalence}, we are regarding $z^*(\cdot)$ as a solution of \eqref{eqn:bar_F} with initial condition $z^*(\tau_0)=x_0$).
By the optimality principle, for all $\tau\in[\tau_0,S^*]$,  $$V(\tau,z^*(\tau))=V(\tau_0,z(\tau_0))\leq\beta_0.$$
In view of Proposition \ref{prop:Lips}, $(\tau, z^*(\tau))\in \D$  for all $\tau\in[\tau_0,S^*]$. Furthermore, one has that $\tau_0<\mathrm{Esc}((\cdot, z^{*}(\cdot), V((\cdot, z^{*}(\cdot))), \Omega^{\mathsf{c}}\times\mathbb{R})\leq S^{*}$ since $V(S^*,z^*(S^*))=W(S^*,z^*(S^*))$.
For all $t\geq 0$, define $\tau(t)=t+\tau_0$. Therefore $\dot{\tau}(t)=1$ and $\tau(0)=\tau_0$. Define $x^*(t)=z^*(\tau(t))$ and observe that
  $y^*(t)=(\tau(t),x^*(t),\beta_0)$ is a solution of \eqref{eqn:gamma} with initial conditions $(\tau(0), x^*(0), \beta(0))=(\tau_0, x_0, \beta_0)$. Hence, setting $T^*:=S^*-\tau_0> 0$, one has $$V(\tau(t),x^*(t))\leq\beta_0,$$ for all $t\in[0,T^*]$. Furthermore, one has that $(\tau(t),x^*(t))\in \D$ for all $t\in[0,T^*]$ and that $0<\mathrm{Esc}((\tau(\cdot), x^{*}(\cdot), V(\tau(\cdot),x^*(\cdot))), \Omega^{\mathsf{c}}\times\mathbb{R})\leq T^{*}$ since $V(\tau(T^*), x^*(T^*))=W(\tau(T^*), x^*(T^*))$. This concludes the proof.
\end{proof}
\begin{proposition}\label{prop:strong}
\emph{Assume that $H_1$-$H_4$ are satisfied and that $V$ is upper semi-continuous. Define $\mathcal{H}=\mathrm{hypo}(V)$, that is}
\begin{displaymath}
  \mathrm{hypo}(V)=\{(\tau,x,\beta)\in \mathbb{R}^{1+n+1}: \,V(\tau,x)\geq\beta\}.
\end{displaymath}
\emph{Then $(\mathcal{H},\Gamma)$ is strongly invariant}.
\end{proposition}
\begin{proof}{\emph{Proof.}}
Since $V$ is upper semi-continuous, then $\mathcal{H}$ is a closed set. 
Fix $(\tau_0,x_0,\beta_0)\in \mathcal{H}$. If $(\tau_0,x_0)\notin \D$, then the thesis is trivially satisfied. So let us assume that $(\tau_0,x_0)\in \D$. Then $V(\tau_0,x_0)\geq \beta_0$. In view of Remark \ref{rem:augmented}, given any trajectory of \eqref{eqn:gamma}, namely $y(t)=(\tau(t),x(t),\beta_0)$, with initial condition $y(0)=(\tau_0,x_0,\beta_0)$, it is possible to define $t(\tau)=\tau-\tau_0$ and $z(\tau)=x(t(\tau))$, trajectory of $\dot{z}(\tau)\in\bar{F}(\tau,z(\tau))$ a.e. $\tau\in[\tau_0,+\infty)$, with initial  condition $z(\tau_0)=x_0$. Let us observe  that the value function $V$ is non decreasing along $z(\tau)$ so that, for all $S\geq\tau_0$ and $\tau\in[\tau_{0},S]$, one has
 $$ \beta_0\leq V(\tau_0,x_0)\leq V(\tau,z(\tau)).$$
Finally, for any solution $y(t)$ of \eqref{eqn:gamma} with initial condition $y(t_0)=(\tau_0,x_0,\beta_0)$ and for any $T\geq 0$, one can set $S=T+\tau_0\geq\tau_0$. Hence 
$$ \beta_0\leq V(\tau_0,x_0)\leq V(\tau(t),x(t)),$$
for all $t\in[0,T]$. This concludes the proof.
\end{proof}
 \section{Hamilton-Jacobi-Bellman Inequalities}\label{sec:HJB}
\label{hj}
In this section we characterise the value function \eqref{eqn:value-function} as the unique viscosity solution of the Hamilton-Jacobi related to the problem $(P)$. Let us now define the maximized and minimized  Hamiltonians for $\Gamma$.
\begin{definition}
\emph{Fix $\eta, y \in\R^{1+n+1}$, $y=(\tau,x,a)$, $\eta=(\eta_1, \eta_2, \eta_3)$.
The minimized Hamiltonian is defined as}
\begin{displaymath}
    h_{\Gamma}(y,\eta)=\min_{v\in\Gamma(y)}\left<v,\eta\right>=\min_{v\in\bar{F}(\tau,x)}\left<(1,v,0), (\eta_1,\eta_2,\eta_3) \right>;
\end{displaymath}
\emph{the maximized Hamiltonian is defined as}
\begin{displaymath}
    H_{\Gamma}(y,\eta)=\max_{v\in\Gamma(y)}\left<v,\eta\right>=\max_{v\in\bar{F}(\tau,x)}\left<(1,v,0), (\eta_1,\eta_2,\eta_3) \right>.
\end{displaymath}
\end{definition}
We will now state the weak invariance and strong invariance characterisations in Hamiltonian forms. Let us first observe that,
since the set-valued map $\Gamma$ inherits the same properties of $\bar{F}$ (summarised in Proposition \ref{prop:properties} and hypothesis $H_5$), then the following result holds true (\cite{wolenski1998proximal}, Theorem $3.1$):
\begin{proposition}\label{prop:clarke}
\emph{Assume $H_1$-$H_5$ are satisfied, $V$ is  lower semi-continuous and $\mathcal{O}\subseteq \mathbb{R}^{1+n+1}$ is an open set. Then the following statements are equivalent:
\begin{itemize}
    \item[$i)$] $(\mathrm{epi}(V),\Gamma)$ is weakly invariant in $\mathcal{O}$;
    \item[$ii)$] For all $(\tau,x,a)\in \mathrm{epi}(V)\cap \mathcal{O}$,\\ $h_\Gamma((\tau,x,a),\eta)\leq 0$ for all $\eta \in N_{\mathrm{epi}(V)}^P(\tau,x,a)$.
\end{itemize}}
\end{proposition}
Since $\Gamma$ also satisfies the relation \eqref{eqn:donchev_osl}, then one can invoke a strong invariance principle proved in \cite{donchev2005strong} and (\cite{donchev2008extensions}, Theorem 8).
\begin{proposition}\label{prop:donchev}
\emph{Assume $H_1$-$H_5$ are satisfied and $V$ is upper semi-continuous. Then the following statements are equivalent:
\begin{itemize}
    \item[$i)$] $(\mathrm{hypo}(V),\Gamma)$ is strongly invariant
    \item[$ii)$] For all $(\tau,x,a)\in \mathrm{hypo}(V)$ and $\eta \in N_{\mathrm{hypo}(V)}^P(\tau,x,a)$,\\ $\limsup_{(\tau',x',a')\to_{\eta}(\tau,x,a)}H_\Gamma((\tau',x',a'),\eta)\leq 0$.
\end{itemize}}
\end{proposition}
In the previous Proposition, given $z\in\R^{1+n+1}$ and a non-zero vector $\eta\in\R^{1+n+1}$, $z'\to_{\eta}z$ is equivalent to say that $z'\to z$ and $(z'-z)/|z'-z|\to\eta/|\eta|$.

In what follows, we prove a comparison principle result characterizing any continuous function that exhibits  the same qualitative properties of the value function $V$. Precisely:
\begin{proposition}\label{prop:theta}
\emph{Assume that \textbf{(SH)} hold and that \textbf{(GC)} and \textbf{(IPC)} are satisfied. Let $\Theta:\mathbb{R}^{1+n}\rightarrow \mathbb{R}$ be a continuous, bounded below function such that:
\begin{itemize}
    \item[a)] $\Theta(t,x)\leq W(t,x)$ for each $(t,x)\in \mathrm{Gr}\,\T$;
    \item[b)] $\Theta(t,x)=+\infty$ for all $(t,x)\not\in\D$;
    \item[c)] $\Theta(t_k,x_k)\rightarrow+\infty$ for all $(t_k,x_k)\in\D$ such that $(t_k,x_k)\rightarrow(t,x)\in\mathrm{Bd}(\D)$;
    \item[d)] For every $(t_k, x_k)\in \D$ such that $t_k\rightarrow \infty$, then $\Theta(t_k,x_k)\rightarrow \infty$.    \end{itemize}
    Then one has that:
    \begin{itemize}
    \item [i)] If $(\mathrm{epi}(\Theta), \Gamma)$ is weakly  invariant in $\mathcal{O}=\left(\Omega^c\cap \D \right)\times \R$, where
$$\Omega=\left\{(\tau,x)\in\mathrm{Gr}\,\T:\ \Theta(\tau,x)=W(\tau,x) \right\},$$
then one has that $V(t,x)\leq \Theta(t,x)$. 
\item [ii)] if $(\mathrm{hypo}(\Theta), \Gamma)$ is strongly invariant. Then one has \textcolor{black}{that} $V(t,x)\geq\Theta(t,x)$.
\end{itemize}
}
\end{proposition}
\begin{proof}{\emph{Proof.}}
i). Given any $y_0=(t_0,x_0,a_0)\in \mathrm{epi}(\Theta)\cap \mathcal{O}$, let us define
 \begin{equation}\label{T_max}
\begin{array}{lr}
T_{max}:=\sup\, \left\{T>0:\; \exists \; y(\cdot)\; \mathrm{solution}\;\mathrm{of}\;(AD)_{y_0}\right . \\[1mm]
\left . \qquad \qquad\qquad\; \mathrm{s.t.}\; y(t)\in \mathrm{epi}(\Theta)\cap \mathcal{O} \;\; \mathrm{for}\;\mathrm{all}\; t\in[0,T)\right\}.
\end{array}
\end{equation}
Since the couple $(\mathrm{epi}(\Theta),\Gamma)$ is weakly invariant in $\mathcal{O}$, the supremum in \eqref{T_max} is taken over a non-empty set. In what follows, we will show that the supremum in \eqref{T_max} is actually a maximum. In fact, let us take a maximizing sequence of trajectories $y_n(\cdot)$ of $(AD)_{y_0}$ and the related $T_n$ such that $y_n(t)\in \mathrm{epi}(\Theta)\cap \mathcal{O}$ for all $t\in [0,T_n)$. If $T_n\rightarrow \infty$,  then one would easily get  a contradiction from  the weak invariance of the couple $(\mathrm{epi}(\Theta), \Gamma)$ and from the condition d) on $\Theta$. Furthermore, by standard compactness arguments (see, e.g. Proposition $2.6.1$, \cite{vinter2010optimal}), $y_n(\cdot)\rightarrow y(\cdot)$ uniformly on $[0,T_{max}]$, where $y(\cdot)$ is a  trajectory of $(AD)_{y_0}$.  This implies that the supremum in \eqref{T_max} is a maximum and that there exists a solution $y(t)=(\tau(t),x(t),a_0)$ to \eqref{eqn:gamma} with initial condition $y(0)=(t_0,x_0,a_0)$ (where $\tau(t)=t+t_0$) such that  $\Theta(\tau(t),x(t))\leq a_0$ for all $t\in[0, T_{max})$. 
Fix $a_0=\Theta(t_0,x_0)$ and let us now show that $T_{max}=\mathrm{Esc}(y(\cdot), \mathcal{O})$.  In fact, if $T_{max}\neq \mathrm{Esc}(y(\cdot), \mathcal{O})$, this implies that $y(t)\in \mathrm{epi}(\Theta)\cap \mathcal{O}$ for all $t\in [0, T_{max}]$ and that $y(T_{max}+\varepsilon)\notin \mathrm{epi}(\Theta)$, $y(T_{max}+\varepsilon)\in \mathcal{O}$ for every $\varepsilon>0$ sufficiently small.  However, using again  the weak invariance principle of the couple $(\mathrm{epi}(\Theta),\Gamma)$, one could  construct a new trajectory $\tilde{y}(\cdot)=(\tilde{\tau}(\cdot), \tilde{x}(\cdot), \tilde{a})$ defined on $[T_{max}, \tilde{T})$ for some $\tilde{T}>T_{max}$,  with initial condition $\tilde{y}(T_{max})=y(T_{max})$ such that $\Theta(\tilde{\tau}(t),\tilde{x}(t))\leq \Theta(\tau(T_{max}), x(T_{max}))=\Theta(t_0,x_0)$
for every $t\in [T_{max}, \tilde{T})$, which is clearly a contradiction with the definition of $T_{max}$. These arguments show that  there exists a trajectory $y(\cdot)$ solution of \eqref{eqn:gamma} such that $y(0)=(t_0,x_0,a_0)$ and $y(t)\in \mathrm{epi}(\Theta)\cap \mathcal{O}$ for every $t\in[0, \mathrm{Esc}(y(\cdot), \mathcal{O})=T_{max})$. In particular this implies that
\begin{equation}\label{eqn:weak_theta}
\Theta(\tau(t),x(t))\leq \Theta(t_0,x_0),
\end{equation}
for every $t\in[0, T_{max})$. Let us now observe that, in view of condition c) on $\Theta$ and on relation \eqref{eqn:weak_theta}, one has that $T_{max}=\mathrm{Esc}(y(\cdot), \mathcal{O})=\mathrm{Esc}(y(\cdot), \Omega^{c}\times \R)$. One then easily obtains that $(\tau(T_{max}),x(T_{max}))\in \mathrm{Gr}\T$ and that
\begin{equation}
     W(\tau(T_{max}),x(T_{max}))=\lim_{t\rightarrow T_{max}}\Theta(\tau(t),x(t))\leq \Theta(t_0,x_0).
\end{equation}
Hence $V(t_0,x_0)\leq \Theta(t_0,x_0)$.

ii). The couple $(\mathrm{hypo}(\Theta),\Gamma)$ is strongly invariant. This implies  that, given any $(t_0,x_0,a_0)\in \mathrm{hypo}(\Theta)$ and $T\geq 0$, any solution $y(t)=(\tau(t),x(t),a_0)$ of \eqref{eqn:gamma} with initial condition $y(0)=(t_0,x_0,a_0)$ remains in $\mathrm{hypo}(\Theta)$ for all $t\in [0, T]$.
If $(t_0,x_0)\not\in\D$, then $+\infty=V(t_0,x_0)=\Theta(t_0,x_0)$ in view of condition b) on $\Theta$. 
Fix now $(t_0,x_0)\in\D$ and take $(\bar{S},\bar{x})\in\A_{(t_0,x_0)}$. Then there exists a solution $z(\tau)$ such that $\dot{z}(\tau)\in\bar{F}(\tau,z(\tau))$ for all $\tau\in[t_0,\bar{S}]$, $z(t_0)=x_0$ and $z(\bar{S})=\bar{x}$. Define $\tau(t)=t+t_0$, $x(t)=z(\tau(t))$ and $T=\bar{S}-t_0$. Then $y(t)=(\tau(t),x(t),a_0)$ is a solution of \eqref{eqn:gamma} for all $t\in[0,T]$ with initial condition $y(0)=(t_0,x_0,a_0)$.
Let us choose $a_0=\Theta(t_0,x_0)$. It follows from the strong invariance principle and the condition a) on $\Theta$ that
\begin{align*}
W(\bar{S},\bar{x})\geq\Theta(\tau(T),x(T))\geq a_0=\Theta(t_0,x_0)
\end{align*}
Hence $\Theta(t_0,x_0)\leq V(t_0,x_0)$. This concludes the proof.
\end{proof}
The results in the previous sections  permit to characterize the value function as the unique continuous, viscosity  solution of a set of Hamilton-Jacobi inequalities. Indeed one has the following:
\begin{theorem}\label{thm:char_cone}
\emph{Assume hypotheses \textbf{(SH)} and that conditions \textbf{(GC)}, \textbf{(IPC)} are satisfied. Then the value function $V$  is the unique continuous, bounded below, locally Lipschitz in $\mathcal{D}$ function which satisfies the following properties: 
\begin{itemize}
\item[i)] $V(t,x)\leq W(t,x)$ for each $(t,x)\in \mathrm{Gr}\,\T$;
\item[ii)] $V(t,x)=+\infty$ for all $(t,x)\not\in\D$;
\item[iii)] $V(t_k,x_k)\rightarrow+\infty$ for all $(t_k,x_k)\in\D$ such that $(t_k,x_k)\rightarrow(t,x)\in\mathrm{Bd}(\D)$;
\item[iv)] For every $(t_k, x_k)\in \D$ such that $t_k\rightarrow \infty$, then $V(t_k,x_k)\rightarrow \infty$; 
\end{itemize}
Let us consider the no-characteristic set
$$\Omega=\left\{(\tau,x)\in\mathrm{Gr}\,\T:\ V(\tau,x)=W(\tau,x) \right\}.$$
Then:
\begin{itemize}
\item[v)] take $\mathcal{O}:=\Omega^{c}\times \mathbb{R}$. For every $(t,x,a)\in\mathrm{epi}(V)\cap \mathcal{O}$ one has
\begin{equation}
\min_{v\in \bar{F}(t,x)}(1,v,0)\cdot p\leq 0\qquad\qquad 
    \forall p\in N^P_{\mathrm{epi}(V)}(t,x,a);
\end{equation}
\item[vi)] for every $(t,x,b)\in\mathrm{hypo}(V)$, one has
\begin{equation}
\limsup_{(t',x',b')\to_{p}(t,x,b)}\max_{v\in \bar{F}(t',x')}(1,v,0)\cdot p\leq 0, 
\end{equation}
for all $p\in N^P_{\mathrm{hypo}(V)}(t,x,b)$;
\item[vii)] for every $(t,x)\in\D\cap (\Omega^c)$ one has 
\begin{equation}
p_t+\min_{v\in \bar{F}(t,x)}v\cdot p_x\leq 0,\quad \forall\ (p_t,p_x)\in\partial_P V(t,x);
\end{equation}
\item[viii)] for every $ (t,x)\in\D$, one has
\begin{equation}
q_t+\liminf_{x'\rightarrow_{-q_x}x} \left\{\min_{v\in \bar{F}(t,x')}v\cdot q_x\right\}\geq 0,
\end{equation}
for every $ q=(q_t,q_x)\in\partial^P V(t,x)$.
\end{itemize}}
\end{theorem} 
\begin{proof}{\emph{Proof.}}
Conditions i)-vi) follow from Propositions \ref{prop:gc}, \ref{prop:Lips}-\ref{prop:theta}.
Furthermore, in view of Proposition \ref{prop:Lips} and Theorem \ref{prop:value_lips}, $V$ is locally Lipschitz continuous \textcolor{black}{on} $\mathcal{D}$ and continuous \textcolor{black}{on} $\mathbb{R}^{1+n}$. Theorem \ref{prop:free} assures that $V$ is bounded below. If either $\partial_P V(t,x)=\emptyset$ or $\partial^P V(t,x)=\emptyset$, then, respectively, condition vii) and condition viii) are satisfied. It remains to show conditions vii)-viii) in the other cases to conclude the proof.
It follows from an easy application of  (\cite{vinter2010optimal}, Proposition $4.3.4$)  that, for every $(t,x)\in\D$, one has 
\begin{equation}
\begin{array}{ll}
  N^P_{\mathrm{epi}(V)}(t,x,V(t,x))=\qquad\\\qquad \{(\lambda p,-\lambda):\,\lambda>0,\quad p\in\partial_P V(t,x)\}\cup\{(0,0)\},\\[2mm]
  N^P_{\mathrm{hypo}(V)}(t,x,V(t,x))=\qquad\\\qquad \{(-\lambda q,\lambda):\,\lambda>0,\quad q\in\partial^P V(t,x)\}\cup\{(0,0)\}.
\end{array}
\end{equation} 
Let $p=(p_t,p_x)\in\partial_P V(t,x)$. Then $(\lambda p,-\lambda)\in N^P_{\mathrm{epi}(V)}(t,x,V(t,x))$ for every $\lambda>0$.  It follows from condition v) that, rescaling  w.r.t. $\lambda>0$, one obtains
 \begin{align*}
    &\min_{v\in \bar{F}(t,x)}(1,v,0)\cdot (\lambda p, -\lambda) \leq 0\implies\\
    &\lambda \min_{v\in \bar{F}(t,x)}(1,v)\cdot (p_t,p_x)\leq 0\\
    &\implies p_t+\min_{v\in \bar{F}(t,x)}v\cdot p_x\leq 0 ,
\end{align*}
for every $(t,x)\in\D\cap (\Omega)^{c}$, for every $(p_t,p_x)\in\partial_P V(t,x)$. 

Similarly, if $q=(q_t,q_x)\in\partial^P V(t,x)$ then $(-\lambda q,\lambda)\in N^P_{\mathrm{hypo}(V)}(t,x,V(t,x))$ for every $\lambda>0$. Hence, setting $\bar{q}=(-\lambda q,\lambda)\in N^P_{\mathrm{hypo}(V)}(t,x,V(t,x))$, it follows from condition vi) that
\begin{align}
    &\limsup_{(t',x',a')\rightarrow_{\bar{q}}(t,x,V(t,x))}\max_{v\in \bar{F}(t',x')}(1,v,0)\cdot(-\lambda q,\lambda)\leq 0\implies\\
    &\limsup_{(t',x')\rightarrow_{-q}(t,x)}\max_{v\in \bar{F}(t',x')}(1,v)\cdot(-\lambda q_t,-\lambda q_x)\leq 0\implies\\
    &\limsup_{(t',x')\rightarrow_{-q}(t,x)}\{-\lambda \min_{v\in \bar{F}(t',x')}(1,v)\cdot (q_t, q_x)\}\leq 0\implies\\
    &\liminf_{(t',x')\rightarrow_{-q}(t,x)}\left\{q_t+\min_{v\in \bar{F}(t',x')}v\cdot q_x\right\}\geq 0, \label{eqn:super_sol}
\end{align}
for every $(t,x)\in\D$ and $(q_t,q_x)\in\partial^P V(t,x)$. Since the map $t\leadsto\bar{F}(t,x)$ is locally Lipschitz continuous for each $x$, one easily obtains condition viii) from \eqref{eqn:super_sol}. This concludes the proof.
\end{proof}

\begin{remark}\label{rem:final}
Let us make some further comments on the implications of Theorem \ref{thm:char_cone}. Assume that the multifunction $\bar{F}:\mathbb{R}^{1+n}\leadsto \mathbb{R}^n$ is continuous \textcolor{black}{on} $\mathcal{D}$. Then the condition viii), Theorem \ref{thm:char_cone} becomes the usual inequality 
$$q_t+\min_{v\in \bar{F}(t,x)}v\cdot q_x \geq 0,\qquad \forall\;\; (q_t,q_x)\in\partial^P V(t,x).$$
Under these assumptions, the generalised solution characterised by conditions vii)-viii), Theorem \ref{thm:char_cone}, can be interpreted as the classic notion of \textit{viscosity solution} in $\mathcal{D}$ (see e.g. \cite{clarke1995qualitative}). Indeed, fix $(t,x)\in \D$ and assume that $V$ is differentiable at $(t,x)\in\mathcal{D}$. When $(t,x)\in \mathcal{D}\setminus (\mathrm{Gr}\,\T)$, conditions vii)-viii) provide the Hamilton-Jacobi equation
\begin{equation}\label{HJ1}
\partial_t V(t,x)+ \min_{v\in \bar{F}(t,x)}\left\{v\cdot \nabla_x V(t,x)\right\}=0
\end{equation}
and, when $(t,x)\in \mathrm{Gr}\T$, conditions vii)-viii) together with condition i) yield the following Hamilton-Jacobi equation
\begin{equation}\label{HJ2}
\min\left\{(W-V)(t,x), \partial_t V(t,x)+ \min_{v\in \bar{F}(t,x)}v\cdot \nabla_x V(t,x)\right\}=0.
\end{equation}
While the  equation \eqref{HJ1} reflects the need of hitting the target $\mathrm{Gr}\T$ (as it happens in the minimum time problem), the equation \eqref{HJ2} is motivated by the search of a minimum point of $W$ in  $\mathrm{Gr}\T$. \textcolor{black}{As a matter of fact, the obtained Hamilton-Jacobi relations related to the free-time optimal control problem with moving target $(P)$ are distinct from the standard Hamilton-Jacobi relations obtained for fixed time optimal control problems (see, e.g. \cite{vinter2010optimal}). Indeed, the relations i)-viii) of Theorem \ref{thm:char_cone} boil down to solving a different partial differential equation with different boundary conditions.}

   It is also interesting to observe that, when $\bar{F}$ is continuous \textcolor{black}{on} $\mathcal{D}$, then one can prove, following the \textcolor{black}{same} approach \textcolor{black}{outlined} in \cite{donchev2015value}, that conditions  vii)-viii) of Theorem \ref{thm:char_cone} hold true even if the value function $V$ is just continuous. \textcolor{black}{Indeed, the approach in \cite{donchev2015value} is based on a use of the Rockafellar's Horizontal Theorem to represent the horizontal vectors to $\mathrm{epi}(V)$ (or $\mathrm{hypo} (V)$) and a passage to limit through the Hamiltonians in order to obtain conditions  vii)-viii), Theorem \ref{thm:char_cone}.  But such a procedure fails when the dynamics $\bar{F}$ (and so the Hamiltonian) is discontinuous.}

\textcolor{black}{When the multifunction  $\bar{F}:\mathbb{R}^{1+n}\leadsto \mathbb{R}^n$ is merely upper-semicontinuous,  then also the equivalence result between different notions of generalized solutions provided in \cite{clarke1995qualitative}  fails.} In fact, the proofs in \cite{clarke1995qualitative} rely on a density argument among different sub/super-differentials and \textcolor{black}{(again)} in a simple limit taking through the continuous Hamiltonian. Of course, this last step breaks down when the Hamiltonian is discontinuous. 

As a consequence, different definitions of viscosity solutions have been introduced in the literature according to the specific problem needs. For instance, in \cite{giga2013hamilton}, the authors use a definition of viscosity solution based on the notion of limiting sub-differential, while in \cite{serea2003reflecting} a notion of modified viscosity solution is provided to take into account the specific discontinuity arising in reflecting boundary optimal control problems.  The notion of viscosity solution employed in  v)-vi), Theorem \ref{thm:char_cone} is closely related to the one provided in \cite{frankowska1995measurable}, Theorem 5.8, for Hamilton-Jacobi equations with Hamiltonian measurable in time. However, let us stress that Theorem \ref{thm:char_cone} deals with an optimal control problem with ``state" discontinuity, rather than ``time" discontinuity. In this sense, conditions vii)-viii), Theorem \ref{thm:char_cone}, provide important information on the direction along which the Hamiltonian inequalities hold. This will be further discuss in the next Section. Furthermore, it is important to mention that conditions vii)-viii), Theorem \ref{thm:char_cone} rely on the local Lipschitz continuity of the value function $V$ and could not be obtained if $V$ were merely continuous.
\end{remark}

\section{A toy example}\label{sec:ex}
Let us consider the following optimal control problem
\begin{equation}
(P_{ex})
    \begin{cases}
    \mathrm{Minimize}\ W(T,v(T))\\
    \dot{v}(t)\in\ u-\frac{u^2}{2}\partial_v\varphi(v),\ \mathrm{a.e.}\ t\in[t_0,T]\\
    u(t)\in[-2,2]\qquad \mathrm{a.e.}\ t\in[t_0,T]\\
    v(t_0)=v_0\in\R,\\
    (T,v(T))\in\mathrm{Gr}\T\subseteq \R^2
    \end{cases}
\end{equation}
in which $t_0\in \R$,    $\T: \R\leadsto \R^{+}$, $\T(t)=[r,+\infty)$, $W(t,v)=Ct+(1/v^2)$ for some constants $C$, $r>0$ such that $Cr^3\leq1$ and
\begin{align*}
    &\varphi(v)=\begin{cases}
    &0\qquad v<0\\
    &v\qquad v\geq0
    \end{cases} .
\end{align*}
It is easy to see that $(P_{ex})$ is a special case of the general optimal control problem $(P)$, in which the data are defined as $g(t,v,u)=u$, $k(t,v,u,\alpha)=u^2\alpha$, $A=[0,1]$, $\mu(\alpha)=\delta_{\frac{1}{2}}(\alpha)$ and
\begin{align*}
    &F(t,v,u)=u-\frac{u^2}{2}\partial_v\varphi(v).
    \end{align*}
Furthermore, a straightforward computation shows that
\begin{align*}
    &\bar{F}(t,v)=\cup_{u\in U}F(t,v,u)=
    \begin{cases}
    &[-4,\ \frac{1}{2}]\qquad v>0\\
    &[-4,\ 2]\qquad v=0\\
    &[-2,\ 2]\qquad v<0
    \end{cases},
\end{align*}
and that hypothesis $H_1$-$H_7$ and condition \textbf{(GC)} are satisfied (notice that the Lipschitz continuity of the cost function $W$ is required merely in  a neighborhood of the target). Furthermore, by definition of $t\leadsto\T(t)$, $(\bar{t},\bar{v})\in\mathrm{Bd}(\mathrm{Gr}\,\T)$ if and only if $(\bar{t},\bar{v})=(\bar{t},r)$ for some $\bar{t}\in \R$. It follows that $(l^0,l)\in N_{\mathrm{Gr}\T}^P(\bar{t},r)$, $|(l^0,l)|=1$  if and only if $(l^0,l)=(0,-1)$. Hence, \textbf{(IPC)} is satisfied because, for any compact $G\subseteq\R^{2}$ and any $(\bar{t},r)\in\mathrm{Bd}(\mathrm{Gr}\,\T)\cap G$, one has 
$$\min_{\xi\in\bar{F}(\bar{t},r)}(l^0+\left<l,\xi\right>)=-\frac{1}{2}<0$$
The problem $(P_{ex})$ is describing the velocity $v(t)$ of an object that is moving, assuming that the friction acts only in one direction (see Figure \ref{fig2}). When the velocity is negative, 
one can choose the control without taking into account the effect of the friction. When the velocity is positive, the friction  reduces the velocity and one has to choose the control providing the maximum of the difference between velocity and friction. 
The target describes a minimum velocity requirement for the optimal solution of the problem.
\begin{figure}[h!]
    \centering
    \includegraphics[width=0.4\textwidth]{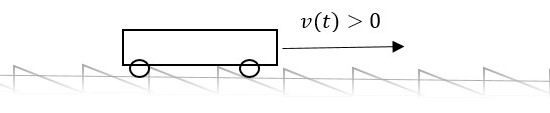}
    \caption{Example of a one side friction framework}
    \label{fig2}
\end{figure}\\
It is natural to guess that the optimal control will be $\bar{u}(t)=2$ for all $t$ such that $v(t)\leq0$. 
Furthermore, when $v(t)>0$ and $v(t)<r$, the optimal control should be positive (to reach the target region) and related to  the maximum velocity  (in order to minimise the time in the cost function). Hence
\begin{align*}
&\partial_u\left\{u-\frac{u^2}{2}\partial_v\varphi(v)\right\}=\partial_u\left\{u-\frac{u^2}{2}\right\}=0 \\
&\rightarrow\ |u|=1\ \rightarrow\ \bar{u}(t)=1.
\end{align*}
Furthermore, $(P_{ex})$ is an example in which the optimal solution does not stop as soon as the target is reached.
Let us study the behaviour of $W(t,\bar{v}(t))$, when $\bar{v}(t)>0$ and $\dot{\bar{v}}(t)\in F(t,\bar{v},\bar{u})=F(t,\bar{v},1)=\frac{1}{2}$:
$$\frac{d}{dt}W(\bar{t},v(\bar{t}))=C-\frac{1}{\bar{v}(\bar{t})^3}=0\ \leftrightarrow\ \bar{v}(\bar{t})=v^*=\sqrt[3]{\frac{1}{C}}\geq r$$
where the last inequality holds since $Cr^3\leq 1$. On the other hand, if  $Cr^3$ were larger than $1$, then $v^*\not\in\T(t)$ and the optimal solution would stop as soon as it reaches the $\mathrm{Gr}\,\T$.
The previous analysis shows that the  guessed optimal control is
\begin{align*}
    \bar{u}(t)=\begin{cases}
    &2\qquad v(t)<0\\
    &1\qquad 0\leq v(t)< v^*\\
    &0\qquad v(t)\geq v^*
    \end{cases},
\end{align*}
and, for any $(t_0,v_0)\in \R^2$, one can also guess that  the value function is
\begin{equation}\label{V_ex}
    V(t_0,v_0)=\begin{cases}
    &Ct_0+\frac{1}{v_0^2}\qquad\qquad\qquad\qquad\ v_0\geq v^*\\
    &2Cv^*-2Cv_0+Ct_0+\frac{1}{{v^*}^2}\quad 0\leq v_0<v^*\\
    &2Cv^*-C\frac{v_0}{2}+Ct_0+\frac{1}{{v^*}^2}\quad\ v_0<0
    \end{cases} .
\end{equation}
We will now  verify that the Hamilton-Jacobi inequalities of Theorem \ref{thm:char_cone} are satisfied by the guessed value function \eqref{V_ex}. In fact, if $V$ is differentiable at $(t_0,v_0)\in\R^2$, then $\partial_tV(t_0,v_0)=C$ and
\begin{align*}
    &\partial_xV(t_0,v_0)=\begin{cases}
    &-\frac{2}{v_0^3}\qquad\ v_0\geq v^*\\
    &-2C\qquad 0< v_0\leq v^*\\
    &-\frac{C}{2}\qquad\ v_0<0 
    \end{cases}.
\end{align*}
Let us also notice that the value function is differentiable at each point except at the origin.
If $v_0\geq v^*$, one can show that \eqref{HJ2} is satisfied. \textcolor{black}{Indeed}, in this case $\bar{F}(t_0,v_0)=[-4,\ \frac{1}{2}]$, $V(t_0,v_0)=W(t_0,v_0)$ and
\begin{align*}
&\partial_tV(t_0,v_0)+\min_{w\in\bar{F}(t_0,v_0)}(w\cdot\partial_vV(t_0,v_0))=\\
&C+\frac{1}{2}\cdot(-\frac{2}{v_0^3})\geq C-\frac{1}{{v^*}^3}=0.
\end{align*}
When $0<v_0\leq v^*$, then $\bar{F}(t_0,v_0)=[-4,\ \frac{1}{2}]$, $V(t_0,v_0)\leq W(t_0,v_0)$ and
\begin{align*}
&\partial_tV(t_0,v_0)+\min_{w\in\bar{F}(t_0,v_0)}(w\cdot\partial_vV(t_0,v_0))=\\
&C+\frac{1}{2}\cdot(-2C)=0,\end{align*}
showing that both equation \eqref{HJ2} (which is valid  when $v_0$ is in the target, namely $r\leq v_0\leq v^*)$ and equation \eqref{HJ1} (valid outside the target) are satisfied.
If $v_0< 0$, then $\bar{F}(t_0,v_0)=[-2,\ 2]$ and 
\begin{align*}
&\partial_tV(t_0,v_0)+\min_{w\in\bar{F}(t_0,v_0)}(w\cdot\partial_vV(t_0,v_0))=\\
&C+2\cdot(-\frac{C}{2})=0,\end{align*}
showing that \eqref{HJ1} is satisfied.
It remains to check what happens at the points in $(t_0,v_0=0)$, $t_0\in \R$. 
In particular, it is easy to check that $\partial_PV(t_0,0)=\emptyset$ for all $t_0\in\R$. Then it is enough to check that condition viii) of Theorem \ref{thm:char_cone} is verified. Let us observe that 
 \begin{equation}
\partial^PV(t_0,0)=\left\{q\in\R^{1+n}|\ q=(C,q_{v}),\ q_{v}\in\left[-2C,-\frac{C}{2}\right]\right\}\end{equation}
Take $(C,q_{v})\in \partial^PV(t_0,0)$. Since $v'\rightarrow_{-q_v}0$ if and only if $v'>0$ and $\bar{F}(t_0,v')=[-4,\frac{1}{2}]$, relation viii) of Theorem \ref{thm:char_cone} is
\begin{align*}
    &C+\liminf_{v'\rightarrow_-q_{v}0}\left\{\min_{w\in[-4,\frac{1}{2}]}w\cdot q_{v}\right\}=C+\frac{q_{v}}{2}\geq0
\end{align*}
for all $q_{v}\in[-2C,-\frac{C}{2}]$. In view of Theorem \ref{thm:char_cone}, one can conclude that  \eqref{V_ex} is the value function of the optimal control problem $(P_{ex})$. As a further implication of condition viii), Theorem \ref{thm:char_cone}, it is interesting to observe that, for any $t>0$, the optimal feedback acceleration at $w^*(t, 0)=1/2$ (and the related optimal feedback control $u^*(t,0)=1$)   is not obtained as the point in $\bar{F}(t,0)$ which realizes the minimized Hamiltonian:
$$w^{*}(t, 0)\neq \mathrm{arg}\min_{w\in \bar{F}(t,0) }\left\{ q_v\cdot w\right\},\quad \forall q_v\,\in\left[-2C,-\frac{C}{2}\right].$$
However, $w^*(t, 0)=1/2$ is obtained as 
$$w^{*}(t, 0)= \lim_{v'\rightarrow 0^+} \left[\mathrm{arg}\min_{w\in \bar{F}(t,v') } q_v\cdot w\right],\; \forall q_v\,\in\left[-2C,-\frac{C}{2}\right].$$
The information in vii)-viii), Theorem \ref{thm:char_cone} provides the way to compute the optimal feedback control as a limit of the Hamiltonian ``argmin" from the ``right direction".

\begin{remark}\label{rem:opt_syn}
In the case in which the multifunction $\bar{F}:\mathbb{R}^{1+n}\leadsto \mathbb{R}^n$ is Lipschitz continuous \textcolor{black}{on} $\mathcal{D}$, well-known sensitivity relations can be applied in order to link a suitable generalized gradient of the value function $V$ to the optimal feedback control (see, e.g., \cite{vinter2010optimal}, Section 12.5 and \cite{cannarsa2000optimality} for a free time optimal control version). When the multifunction $\bar{F}:\mathbb{R}^{1+n}\leadsto \mathbb{R}^n$ is merely upper semicontinuous, then it is much harder to relate the value function to the optimal control. To our knowledge, such a shortcoming is mainly due to the lack of a set of necessary conditions when the dynamics is discontinuous w.r.t. the state variable. However, as the previous example shows, it is not hard to conjecture that any theory or method of optimal synthesis has to take into account the information provided by conditions vii)-viii) of Theorem \ref{thm:char_cone}.
\end{remark}

\section{Conclusions and Future Works}
The main result of this paper is a characterisation of the value function of a free time optimal control problem subject to a controlled differential inclusion as unique, continuous viscosity solution of a related Hamilton-Jacobi equation. The dynamics arises from a class of systems in which the friction is represented by an averaged, upper semi-continuous, controlled differential inclusion.  Under general assumptions, we show that the dynamic equation is well-posed and that the related optimal control problem admits solutions. 

Several theoretical questions (such as controllability, necessary optimality conditions etc.)  and algorithmic considerations related to the present framework can be  considered as future research directions.
Furthermore, the theory provided in this paper will be useful to describe a wide class of phenomena, in which a mechanical constraint producing friction is concerned.  As a further research direction, we plan to use some of the results provided in this paper to model the morphological growth of living tissues, such as roots, stems and vines, in the spirit of \textcolor{black}{\cite{del2019characterization, tedone2020optimal}.}

\section{Appendix}
 Given a closed set $\mathcal{C}\subseteq\R^{1+n}$ and $v\in\R^{1+n}$, let us use 
$$\mathcal{P}_{\mathcal{C}}(v)=\left\{w\in \mathcal{C}|\ |v-w|=d(v,\mathcal{C})\right\}$$
to denote the set of all projections of $v$ on $\mathcal{C}$. 
According to (\cite{vinter2010optimal}, Proposition $4.2.2$), given $w\in \mathcal{C}$, the proximal normal cone $N_{\mathcal{C}}^P(w)$ is characterised as 
\begin{align*}
    N_{\mathcal{C}}^P(w)=\left\{v\in\R^m|\ \exists\ \alpha>0\ \mathrm{s.t.}\ w\in \mathcal{P}_{\mathcal{C}}(w+\alpha v)\right\}.
\end{align*}
Let us define the lower Dini derivative of a function $\psi:\R^{1+n}\rightarrow\R$ in the direction of $(1,y)$ as
$$D^-\psi(t,x;1,y):=\liminf_{h\rightarrow0^+}\frac{\psi(t+h,x+hy)-\psi(t,x)}{h}.$$
To prove Proposition \ref{prop:claim}, we need the following result:
\begin{proposition}\label{prop:tech}
Assume $H_1$-$H_6$ hold and $\mathrm{Gr}\,\T$ satisfies \textbf{(IPC)}. Then, for any compact set $K\subseteq\R^{1+n}$, there exist $\theta,\ \rho>0$ and $L>0$ such that the multifunction $\Phi:\R^{1+n}\leadsto \R^n$ defined as
\begin{equation}\label{def_Phi}
     \Phi(t,x)=\left\{y\in \bar{F}(t,x)|\ D^-\psi(t,x;1,y)\leq\Psi(t,x)\right\}
\end{equation}
is \textcolor{black}{non-empty}. In the previous equation, for each $(t,x)\in\R^{1+n}$, the functions $\psi(t,x)$ and $\Psi(t,x)$ are defined as:
\begin{align*}
    &\psi(t,x)=d\left((t,x),\mathrm{Gr}\T\right),\\
    &\Psi(t,x)=
    \begin{cases}
    (L_F+L)\psi(t,x)-\rho\ &(t,x)\not\in\mathrm{Gr}\T\mathrm{\ and\ }\\&(t,x)\in \mathcal{O}_{\theta}\\
    +\infty\ &\mathrm{otherwise}
    \end{cases},
\end{align*}
where  $\mathcal{O}_{\theta}:=\left(\mathrm{Gr}\T\cap K\right)+\theta\B$ and $L_F$ is the local one-sided Lipschitz constant of $F$ in $\bar{\mathcal{O}}_{1}$.
\end{proposition}
\begin{proof}{\textit{Proof of Proposition 12. }}
Fix $K\subset \R^{1+n}$ a compact set. Take $\mathcal{O}_{\theta}=(\mathrm{Gr}\T\cap K)+\theta\B$ for some $\theta>0$ and $G=\bar{\mathcal{O}}_1$. Let $[\tau_1,\tau_2]=\mathrm{co}\mathcal{P}_{\R}(G)$.
Fix $\theta$ such that
$$\theta<\min\left\{0.5,\frac{1}{L_G+1}\right\},$$
where $L_G$ is such that $|F(t,x,u)|\leq L_G$ for all $(t,x,u)\in[\tau_1,\tau_2]\times \mathrm{co}\mathcal{P}_{\R^n}(G)\times U$ (notice that $L_G$ exists in view of the hypothesis $H_3$). With the choice of $G=\bar{\mathcal{O}}_1$, let us take $\rho>0$ such that \textbf{(IPC)} holds true. It is a straightforward matter to check that $\Psi(t,x)$ is upper semi-continuous. Furthermore, by possibly reducing the size of $\rho$, it is not restrictive to suppose that there exists $0<\rho<1$ satisfying \textbf{(IPC)} with the choice $G=\bar{\mathcal{O}}_1$.
Let us prove that the multifunction $\Phi:\R^{1+n}\leadsto \R^n$ defined in \eqref{def_Phi} is such that $\Phi(t,x)\neq\emptyset$ for every $(t,x)\in\R^{1+n}$. 
If $(t,x)\in\mathrm{Gr}\T$ or $(t,x)\not\in \mathcal{O}_{\theta}$, then $\Phi(t,x)=\bar{F}(t,x)\neq\emptyset$. Otherwise, define
$$(l^0,l)=\frac{(t-\bar{t},x-\bar{x})}{|(t,x)-(\bar{t},\bar{x})|}\in N_{\mathrm{Gr}\T}^P(\bar{t},\bar{x}),\ |(l^0,l)|=1,$$ 
for any $(\bar{t},\bar{x})\in\mathcal{P}_{\mathrm{Gr}\T}(t,x)$. Being $(\bar{t},\bar{x})$ a projection of $(t,x)$ on $\mathrm{Gr}\T$, then $(\bar{t},\bar{x})\in\mathrm{Bd}(\mathrm{Gr}\,\T)$ and one has
$$|(\bar{t},\bar{x})-(t,x)|=d((t,x),\mathrm{Gr}\T)<\theta<1.$$
In particular, this implies $(\bar{t},\bar{x})\in G\cap\mathrm{Bd}(\mathrm{Gr}\,\T)$. In view of \textbf{(IPC)}, there exists $\bar{\xi}\in \bar{F}(\bar{t},\bar{x})$ (and then $\bar{u}\in U$  for which $\bar{\xi}\in F(\bar{t},\bar{x},\bar{u})$) such that
$$l^0+\left<l,\bar{\xi}\right>\leq-\rho.$$ 
Take any $\xi\in F(t,x,\bar{u})$. Then one can obtain the following estimates:
\begin{equation}\label{eq_app}
\begin{array}{ll}
   &  l^0+\left<l,\bar{\xi}-\xi+\xi\right>=l^0+\left<l,\xi\right>+\left<l,\bar{\xi}-\xi\right>\leq-\rho\implies\\
   &  l^0+\left<l,\xi\right>\leq-\rho+\left<l,\xi-\bar{\xi}\right>=-\rho+\frac{\left<x-\bar{x},\xi-\bar{\xi}\right>}{d((t,x),\mathrm{Gr}\T)}.
    \end{array}
\end{equation}
Since the mapping $t\mapsto F(t,x,u)$ is locally Lipschitz continuous uniformly w.r.t. $(x,u)\in \mathrm{co}\mathcal{P}_{\mathbb{R}^n}(G)\times U$, there exist $L>0$ and $\xi^*\in F(\bar{t},x,\bar{u})$ such that $|\xi^*-\xi|\leq L|\bar{t}-t|$. It is then possible to estimate the right hand side of \eqref{eq_app} as follows:
\begin{align*} 
&\frac{\left<x-\bar{x},\xi-\bar{\xi}\right>}{d((t,x),\mathrm{Gr}\T)}=\\
&\frac{\left<x-\bar{x},\xi-\xi^*\right>}{d((t,x),\mathrm{Gr}\T)}+\frac{\left<x-\bar{x},\xi^*-\bar{\xi}\right>}{d((t,x),\mathrm{Gr}\T)}\leq\\
&\frac{|x-\bar{x}||\xi-\xi^*|}{d((t,x),\mathrm{Gr}\T)}+\frac{L_F|x-\bar{x}|^2}{d((t,x),\mathrm{Gr}\T)}\leq\\
&\frac{L|t-\bar{t}||x-\bar{x}|}{d((t,x),\mathrm{Gr}\T)}+\frac{L_F|x-\bar{x}|^2}{d((t,x),\mathrm{Gr}\T)}\leq\\
&\frac{L|(t,x)-(\bar{t},\bar{x})|^2}{d((t,x),\mathrm{Gr}\T)}+\frac{L_F|(t,x)-(\bar{t},\bar{x})|^2}{d((t,x),\mathrm{Gr}\T)},
\end{align*}
where, in turns, we have used the  one-sided locally Lipschitz property of $x\mapsto F(t,x,u)$, uniformly w.r.t. $(t,u)\in \mathrm{co}\mathcal{P}_{\mathbb{R}}(G)\times U$, and the locally Lipschitz continuity of the mapping $t\mapsto F(t,x,u)$, uniformly w.r.t. $(x,u)\in \mathrm{co}\mathcal{P}_{\mathbb{R}^n}(G)\times U$. This analysis shows that
$$l^0+\left<l,\xi\right>\leq-\rho+(L_F+L)d((t,x),\mathrm{Gr}\T).$$
Let us show that $\xi\in\Phi(t,x)$. Since $(\bar{t},\bar{x})\in\mathrm{Gr}\T$ and $(t,x)\notin\mathrm{Gr}\T$, then $\psi(t,x)>0$ and one can compute 
\begin{align*}
    &D^-\psi(t,x;1,\xi)=\\
    &\liminf_{h\to0^+}\frac{\psi^2(t+h,x+h\xi)-\psi^2(t,x)}{h(\psi(t+h,x+h\xi)+\psi(t,x))}\leq\\
    &\liminf_{h\to0^+}\frac{|(t+h-\bar{t},x+h\xi-\bar{x})|^2-|(t-\bar{t},x-\bar{x})|^2}{h(\psi(t+h,x+h\xi)+\psi(t,x))}=\\
    &\frac{t-\bar{t}+\left<x-\bar{x},\xi\right>}{\psi(t,x)}=l^0+\left<l,\xi\right>\leq\Psi(t,x).
\end{align*}
Thus, $\xi\in\Phi(t,x)$. This completes the proof.
\end{proof}

\begin{proof}{\emph{Proof of Proposition \ref{prop:claim}}.}\label{proof:claim}
Since we are in the same hypothesis of Proposition \ref{prop:tech}, \textcolor{black}{let us fix a compact set $K\subset \R^{1+n}$}. Define $\mathcal{O}_1=\left(\mathrm{Gr}\T\cap K\right)+\B$ and $G=\bar{\mathcal{O}}_1$. Let $[\tau_1,\tau_2]=\mathrm{co}\mathcal{P}_{\R}(G)$ and fix $L_G$ such that $|F(t,x,u)|\leq L_G$ for all $(t,x,u)\in[\tau_1,\tau_2]\times \mathrm{co}\mathcal{P}_{\R^n}(G)\times U$. Define \textcolor{black}{the constants $\theta,\ \rho,\ L,\ L_F$ and the set  $\mathcal{O}_{\theta}$} such that Proposition \ref{prop:tech} is satisfied.
Fix $\e>0$ such that
$$\e<\min\left\{1-\theta(L_G+1),\frac{\theta\rho}{C},\frac{\theta(1-\rho)}{C}\right\},$$
where $C=\exp\bigg(\int_{\tau_1}^{\tau_2}\left(L_F+L\right)ds\bigg)$. Let $(t_0,x_0)\in \bar{\mathcal{O}}_{\e}=(\mathrm{Gr}\T\cap K)+\e\bar{\B}$. 
If $(t_0,x_0)\in\mathrm{Gr}\T$, then \eqref{eqn:claim} is trivially satisfied. Alternatively, observe that $(t_0,x_0)\in \mathcal{\bar{O}}_{\e}\subseteq G$ and $x_0\in\mathcal{P}_{\R^n}(G)$. Define $T=t_0+\theta$. Therefore, for any solution $\tilde{x}(\cdot)$ of \eqref{eqn:dynamics} starting from $(t_0,x_0)$ with control $\tilde{u}(\cdot)$, one has
$$|\tilde{x}(t)-x_0|\leq\int_{t_0}^t|F(s,\tilde{x}(s),\tilde{u}(s))|ds\leq L_G|t-t_0|\leq L_G\theta,$$
and $(t,x(t))\in (t_0,x_0)+(L_G+1)\theta\bar{\B}$ for all $t\in[t_0,T]$. Since $\theta(L_G+1)<1-\e<1$, then $(t,\tilde{x}(t))\in G$ for all $t\in[t_0,T]$. Then, in view of Proposition \ref{prop:tech} and a well known selection theorem (see \cite{veliov1997lipschitz}, Proposition 2.1) there exists a solution to 
\begin{equation}\label{eq_Phi}
\dot{x}(t)\in\Phi(t,x(t)),\ x(t_0)=x_0,\ t\in[t_0,T].
\end{equation}
Le us define $\bar{T}\in[t_0,\infty)$ as the first time such that the given solution $x(\cdot)$ of \eqref{eq_Phi} satisfies either $(\bar{T},x(\bar{T}))\in\mathrm{Gr}\T$ or $(\bar{T},x(\bar{T}))\in\mathrm{Bd}(\mathcal{O}_{\theta})$. If there is not such a  $\bar{T}$, let us fix $\bar{T}:=t_0+\theta=T$. Setting $\tilde{d}(t)=d((t,x(t)),\mathrm{Gr}\T)$,  it follows from  \eqref{def_Phi} that, for all $t\in[t_0,\bar{T}]$,  $\tilde{d}(t)$ satisfies
\begin{align*}
    &\dot{\tilde{d}}(t)=D^-\psi(t,x(t);1,\dot{x}(t))\leq (L_F+L)\tilde{d}(t)-\rho,\\
    &\tilde{d}(t_0)=d((t_0,x_0),\mathrm{Gr}\T)\leq\e .
\end{align*}
The definitions of $\tau_1, \tau_2$ imply that $\tau_1\leq t_0\leq t_0+\theta\leq\tau_2$. Using \textcolor{black}{the} \textcolor{black}{Gr\"onwall's} Lemma, then one easily estimates
\begin{align}
&\tilde{d}(t)\leq\exp\bigg(\int_{\tau_1}^{\tau_2}\left(L_F+L\right)ds\bigg)\tilde{d}(t_0)-(t-t_0)\rho=\\
&C\tilde{d}(t_0)-(t-t_0)\rho<C\e-(t-t_0)\rho \label{s=0}.
\end{align}
Recall that the differential equation for $\tilde{d}(\cdot)$ is studied on $[t_0,\bar{T}]$ so that, since $C\e<C\e/\rho\leq\theta$, one has that $\tilde{d}(t)<\theta$ for all $t\in[t_0,\bar{T}]$. Hence, $(t,x(t))\in \mathcal{O}_{\theta}$ for all $t\in[t_0,\bar{T}]$. In particular, $\tilde{d}(t)=0$ if
 \begin{equation}\label{stima_distanza}
 t=t_0+\frac{C\tilde{d}(t_0)}{\rho}<t_0+\theta\implies t-t_0=\frac{C\tilde{d}(t_0)}{\rho}<\theta.
 \end{equation}
This shows that  $\bar{T}\in[t_0,\infty)$ has to be such that $(\bar{T},x(\bar{T}))\in\mathrm{Gr}\T$ and $(\bar{T},x(\bar{T}))\in\mathcal{A}_{(t_0,x_0)}$. 
Furthermore one can obtain the following estimate:
\begin{align}
&\tilde{d}((t_0,x_0),\mathcal{A}(t_0,x_0))\leq |(t_0,x_0)-(\bar{T},x(\bar{T}))|\leq\\
&|t_0-\bar{T}|+|x_0-x(\bar{T})|\leq \frac{C\tilde{d}(t_0)}{\rho}+L_G|t_0-\bar{T}|\leq\\ &\frac{C(L_G+1)}{\rho}\tilde{d}(t_0)=L_K\tilde{d}(t_0), \label{s==0}
\end{align}
where in the last two steps we have used  the relation \eqref{stima_distanza}. This concludes the proof.
\end{proof}\\
\begin{proof}{\emph{Proof of Theorem \ref{prop:value_lips}}.}
Let us show that, for any $(t_0,x_0)\in\D$, there exist $\e_V,\ L_V>0$ such that for all $(t_1,x_1),\ (t_2,x_2)\in \mathcal{N}_{\e_V}:=(t_0,x_0)+\e_V\B$, one has
$$|V(t_1,x_1)-V(t_2,x_2)|\leq L_V|(t_1,x_1)-(t_2,x_2)|.$$
Let us  use $L_W>0$ to denote the Lipschitz constant of $W$ and with $\e_W>0$ the neighborhood radius in which the Lipschitz property of $W$ is verified.\\
Fix $\rho>0$ and $(t_0,x_0)\in \bar{\mathcal{N}}_{\rho}\subset\D$ . Let us  define
$$\bar{T}:=\sup\,\left\{T:\,(T,x(\cdot))\;\mathrm{minimizer}\;\mathrm{of}\;(P)_{(t,x)},\;(t,x)\in \bar{\mathcal{N}}_{\rho}\right\}.$$
Since $\bar{\mathcal{N}}_{\rho}\subset\D$, then $\bar{T}<\infty$.
 Let $(T^*_0, x^*_0(\cdot))$ be the minimizer of the optimal control problem  $P_{(t_0,x_0)}$.
 In view of the hypothesis $H_2$-$H_3$, for all $(t,x)\in \bar{\mathcal{N}}_{\rho}$, there exists a constant $C>0$ such that
 $$\left|(T^*,x^*(T^*))-(T^*_0,x^*_0(T^*_0))\right|\leq 2\bar{T}C+\rho+\bar{T}=:R,$$  
where $(T^*,x^*(\cdot))$  is the minimizer of the optimal control problem $(P)_{(t,x)}$.
Hence, one can fix both $K=(T^*_0,x^*_0(T^*_0))+R\bar{\B}$ and $\e_K,L_K$ such that the statement of Proposition \ref{prop:claim} is satisfied for all $(t,x)\in(\mathrm{Gr}\T\cap K)+\e_K\bar{\B}$. 
\textcolor{black}{Let us now fix} $\e_V>0$ such that
$$\e_V\leq\min\left\{\rho,\frac{\min\{\e_K,\e_W\}}{2\lambda_{r(\rho)}(\bar{T})(L_K+1)}\right\},$$
where $\lambda_{r(\rho)}(\cdot)$ is the function introduced in Remark \ref{rem:traj} with the choice $r(\rho)=|x_0|+\rho$.
Take $(t_1,x_1),\ (t_2,x_2)\in \mathcal{N}_{\e_V}$ such that $V(t_1,x_1)>V(t_2,x_2)$. Let $u_2(\cdot)$ be the optimal control starting from $(t_2,x_2)$ with trajectory $x_2(\cdot)$ and optimal time $T_2$. 
Then, since $\mathcal{N}_{\e_V}\subseteq \mathcal{N}_{\rho}$, one has that also $(T_2,x_2(T_2))\in\mathrm{Gr}\T\cap K$.
It is convenient to distinguish the two cases $T_2>t_2$ and $T_2=t_2$: 

\textbf{CASE 1:} If $T_2>t_2$, then it is not restrictive to  assume also $T_2>t_1$ (it is sufficient to reduce the size of $\e_V$). Let $x_1(\cdot)$ be the trajectory of \eqref{eqn:dynamics} starting from $(t_1,x_1)$ with control $u_2(\cdot)$. Proposition \ref{prop:principle} ensures $V(t_2,x_2)=W(T_2,x_2(T_2))$ and $V(t_1,x_1)\leq V(T_2,x_1(T_2))$.
In view of \eqref{Lipschitz_2}, one has that
\begin{align*}
&|(T_2,x_2(T_2))-(T_2,x_1(T_2))|\leq\\
&\qquad\leq|(t_1,x_1)-(t_2,x_2)|\lambda_{r(\rho)}(T_2)\leq 2\e_V\lambda_{r(\rho)}(\bar{T})<\e_K,
\end{align*}
where $\lambda_{r(\rho)}(\cdot)$ is the function introduced in Remark \ref{rem:traj} with the choice $r(\rho)=|x_0|+\rho$.
Therefore $(T_2,x_1(T_2))\in(\mathrm{Gr}\T\cap K)+\e_K\bar{\B}$. Let us set $(\bar{t}, \bar{x}):=(T_2,x_1(T_2))$ and take $\bar{x}_1(\cdot)$ be a solution of the differential inclusion \eqref{eq_Phi} starting from $(\bar{t},\bar{x})$. Then, in view of the relation \eqref{s=0}, there exists a time $\bar{T}_1\geq \bar{t}$ such that $(\bar{T}_1,\bar{x}_1(\bar{T}_1))\in\mathrm{Gr}\T$. 
In particular, one obtains the inequality $V(t_{1},x_1)\leq W(\bar{T}_1,\bar{x}_1(\bar{T}_1))$.
It then follows from such a construction the estimate
\begin{equation}\label{stima_base_V}
V(t_1,x_1)-V(t_2,x_2)\leq W(\bar{T}_1,\bar{x}_1(\bar{T}_1))-W(T_2,x_2(T_2)).
\end{equation}
Furthermore, let us observe that
\begin{align}
    &|(\bar{T}_1,\bar{x}_1(\bar{T}_1))-(T_2,x_2(T_2))|\leq  \label{stima1}\\ 
    &|(\bar{T}_1,\bar{x}_1(\bar{T}_1))-(\bar{t},\bar{x})|+|(\bar{t},\bar{x})-(T_2,x_2(T_2))|\leq\\
    &|(\bar{T}_1,\bar{x}_1(\bar{T}_1))-(\bar{t},\bar{x})|+\lambda_{r(\rho)}(\bar{T})|(t_1,x_1)-(t_2,x_2)|,
\end{align}
where we have just used the relation \eqref{Lipschitz_2} in the last inequality with the choice $r(\rho)=|x_0|+\rho$. Using now the relation \eqref{stima_distanza} in \eqref{stima1} (as it is done in \eqref{s==0}), one obtains
\begin{align}
    &|(\bar{T}_1,\bar{x}_1(\bar{T}_1))-(\bar{t},\bar{x})|\leq  \label{stima2}\\
    &L_Kd((\bar{t},\bar{x}),\mathrm{Gr}\T)\leq L_K|(\bar{t},\bar{x})-(T_2,x_2(T_2))|.
\end{align}
So one can obtain from \eqref{stima1} and \eqref{stima2} the relevant estimates
\begin{align}
&|(\bar{T}_1,\bar{x}_1(\bar{T}_1))-(T_2,x_2(T_2))|\leq \label{stima3}\\
&(L_K+1)\lambda_{r(\rho)}(\bar{T})|(t_1,x_1)-(t_2,x_2)|\leq\\ &2\lambda_{r(\rho)}(\bar{T})(L_K+1)\e_V\leq\e_W,
\end{align}
where in the first inequality we have used  again the relation \eqref{Lipschitz_2} with the choice $r(\rho)=|x_0|+\rho$.\\
\textbf{CASE 2}: If $T_2=t_2$ then $|(t_2,x_2)-(t_1,x_1)|\leq2\e_V<\e_K$ and $(t_1,x_1)\in(\mathrm{Gr}\T\cap K)+\e_K\bar{\B}$. Let us set $(\bar{t},\bar{x})=(t_1,x_1)$ and, as in the previous case, take  a solution $\bar{x}_1(\cdot)$  of \eqref{eq_Phi} starting from $(\bar{t},\bar{x})$ and a time $\bar{T}_1\geq \bar{t}$ such that $(\bar{T}_1,\bar{x}_1(\bar{T}_1))\in\mathrm{Gr}\T$. Using the same argument employed in \textbf{CASE 1}, one can obtain the relations \eqref{stima_base_V} and \eqref{stima3}.
Therefore, in both \textbf{CASES 1-2}, the hypothesis $H_7$ can be invoked  and it follows from the estimates \eqref{stima_base_V} and \eqref{stima3} that
\begin{align*}
    &V(t_1,x_1)-V(t_2,x_2)\leq W(\bar{T}_1,\bar{x}_1(\bar{T}_1))-W(T_2,x_2(T_2))\leq\\
    &L_W|(\bar{T}_1,\bar{x}_1(\bar{T}_1))-(T_2,x_2(T_2))|\leq\\
    &L_WC_1|(t_1,x_1)-(t_2,x_2)|:=L_V|(t_1,x_1)-(t_2,x_2)|,
\end{align*}
where $C_1=2\lambda_{r(\rho)}(\bar{T})(L_K+1)$ if $T_2>t_2$ and $C_1=L_K+1$ if $T_2=t_2$. This concludes the proof.
\end{proof}
\bigskip

\emph{Acknowledgments.} This project has received funding from the European Union's Horizon 2020.
Research and Innovation Programme under Grant Agreement No 824074.

\bibliographystyle{IEEEtran}
\bibliography{reference}

\end{document}